\documentclass[11pt,reqno]{amsart}
\usepackage[margin=1.15in]{geometry}
\usepackage{amsmath,amssymb,amsthm,mathtools,booktabs,enumitem}
\usepackage[colorlinks=true,linkcolor=blue,citecolor=blue,urlcolor=blue]{hyperref}

\theoremstyle{plain}
\newtheorem{theorem}{Theorem}[section]
\newtheorem{proposition}[theorem]{Proposition}
\newtheorem{lemma}[theorem]{Lemma}
\newtheorem{corollary}[theorem]{Corollary}
\theoremstyle{definition}
\newtheorem{definition}[theorem]{Definition}
\newtheorem{remark}[theorem]{Remark}

\DeclareMathOperator{\Der}{Der}
\DeclareMathOperator{\Hom}{Hom}
\DeclareMathOperator{\Ext}{Ext}
\newcommand{\kk}{\Bbbk}
\newcommand{\Zt}{Z^2}
\newcommand{\Bt}{B^2}
\newcommand{\Ht}{H^2}
\newcommand{\Fp}{\mathbb{F}_p}
\newcommand{\Zp}{\mathbb{Z}/p}

\begin{document}

\title[Cohomology and extensions of truncated Novikov algebras]
{Cohomology and extensions of Novikov algebras\\ of truncated polynomials}

\author{Hassan Alhussein}
\address{Siberian State University of Telecommunications and Information Sciences,
Novosibirsk, Russia}
\address{Novosibirsk State University of Economics and Management, Novosibirsk, Russia}

\subjclass[2020]{Primary 17A30, 17D25; Secondary 16E40, 16S80, 17B56}

\keywords{Novikov algebra, truncated polynomial algebra, second cohomology,
abelian extension, infinitesimal deformation, positive characteristic,
irreducible module, modular representation}

\begin{abstract}
Let $\kk$ be a \emph{field of characteristic $p>0$, not assumed algebraically closed}, and
let $V=\kk[x]/(x^p)$ be the Novikov algebra with product $a\circ b=ab'$. For $\lambda\in\kk$,
let $M(\lambda)$ be Xu's module. We compute the second cohomology $\Ht(V,M(\lambda))$ for all $\lambda$
and $p$, and describe the associated abelian extensions. The computation utilizes a
$\Zp$-graded presentation $V\cong\kk[t]/(t^p-1)$ to bypass truncation issues and simplify
cocycle identities. We determine the exact dimensions of $\Ht(V,M(\lambda))$, showing it is
$0$ for $\lambda\notin\Fp$, $3$ for $\lambda\in\Fp$ with odd $p$, and $4$ for $\lambda\in\Fp$
with $p=2$. Explicit cocycle representatives are provided for all cases. As corollaries, we
show that every abelian extension of $V$ by $M(\lambda)$ splits when $\lambda\notin\Fp$. We
also treat the characteristic-$0$ analogue $P=\kk[t]$: Xu's parameter $\lambda$ collapses to
the single value $\lambda=0$, and $\Ht(P,M(\lambda))=0$ throughout, so $P$ is rigid (in fact
formally rigid) while its positive-characteristic truncation never is --- within this family,
it is truncation rather than positive characteristic per se that destroys rigidity.
\end{abstract}

\maketitle

\section{Introduction}\label{sec:intro}

\subsection{Background}
Novikov algebras arose independently in two settings: in the study of Hamiltonian operators
and formal variational calculus by Gel'fand and Dorfman \cite{GelfandDorfman}, and in the
theory of Poisson brackets of hydrodynamic type by Balinskii and Novikov
\cite{BalinskiiNovikov}; a further impulse came from Zel'manov's work on local
translation-invariant Lie algebras \cite{Zelmanov}. A \emph{Novikov algebra} is a
right-symmetric (equivalently, right pre-Lie) algebra which is in addition right
commutative; see Section~\ref{sec:prelim} for the identities.

The generic example is the one relevant here: if $A$ is a commutative associative algebra
with a derivation $d$, then $a\circ b:=a\,d(b)$ makes $A$ a Novikov algebra. Structure
theory in this direction was developed by Osborn \cite{Osborn1992a,Osborn1992b,Osborn1994}
and, in prime characteristic, by Xu \cite{Xu1996}, who classified the finite-dimensional
simple Novikov algebras over an algebraically closed field of characteristic $p>0$ together
with their irreducible modules; the characteristic-zero classification is in \cite{Xu2001}.
Related developments include Novikov--Poisson algebras \cite{XuNP}, the classification in
low dimensions \cite{BaiMeng}, solvability and nilpotency questions
\cite{ShestakovZhang}, and, more recently, further work on simple Novikov algebras of
characteristic $p$ \cite{ZhelyabinZakharov}.

The cohomology and deformation theory of right-symmetric algebras --- which contains the
Novikov case --- was developed by Dzhumadil'daev \cite{Dzhu1999}, where cohomologies of
$\mathfrak{gl}_n$ and of half-Witt algebras in characteristic $0$ and $p>0$ are computed and
right-symmetric central extensions of Novikov algebras are constructed; the vanishing
statement Theorem~\ref{thm:vanish} below is the exact analogue, for the Novikov operad and
for the truncated polynomial algebra, of the vanishing results obtained there for half-Witt
algebras outside finitely many resonant weights, and our finite dimension count
$\dim\Ht(V,M(\lambda))\in\{0,3,4\}$ is the corresponding local, finite-dimensional
counterpart of the (typically infinite-dimensional) cohomology computed there for the
infinite-dimensional Witt-type case. The combinatorial side of free Novikov algebras is
treated in \cite{DzhuLofwall}. From the operadic point of view, the complex used below is an
instance, for the Koszul operad $O_{Nov}$ governing Novikov algebras, of the cohomology
theory of algebras over a quadratic operad developed by Balavoine and Kolesnikov
\cite{Balavoine1997,Balavoine1998, kol}, itself built on the Koszul duality of Ginzburg and
Kapranov \cite{GinzburgKapranov}; we give this derivation, together with its extension to
coefficients in an arbitrary bimodule, in Section~\ref{sec:operad}. The deformation-theoretic
reading of $\Ht$ goes back to Gerstenhaber \cite{Gerstenhaber}.

\subsection{The algebra studied here}
Throughout, $\operatorname{char}\kk=p>0$ and
\[
V:=\kk[x]/(x^p),\qquad a\circ b:=ab' ,
\]
the \emph{truncated polynomial Novikov algebra}. This is the simplest nontrivial Novikov
algebra of prime characteristic, and it is the local model for the algebras appearing in
\cite{Xu1996}. In the basis $e_i=x^i$, $0\le i\le p-1$, the product is
$e_i\circ e_j=j\,e_{i+j-1}$, with $e_m=0$ for $m\notin[0,p-1]$.

For $\lambda\in\kk$, Xu's modules $M(\lambda)$ \cite{Xu1996} form the natural
one-parameter family of $p$-dimensional modules over $V$. Our aim is the complete
determination of $\Ht(V,M(\lambda))$, for all $\lambda$ and all $p$, together with the
extensions and deformations that this classifies.

\subsection{A remark on the grading}\label{sub:grading}
One point deserves emphasis at the outset, because it governs the entire computation. In the
basis $e_i=x^i$ with $\deg e_i=i-1$, the algebra $V$ is $\mathbb{Z}$-graded, and it is
tempting to impose on $M(\lambda)$ the same truncation convention ``$m_l=0$ for
$l\notin[0,p-1]$''. This is not legitimate: with that convention the bimodule axioms fail
for $\lambda$ outside one value, so that $\delta^2\ne0$ and the quotient $\Zt/\Bt$ is not
defined. The correct framework uses the presentation
\[
V\cong\kk[t]/(t^p-1),\qquad M(\lambda)=t^{\lambda}\kk[t]/(t^p-1),
\]
legitimate because $(t-1)^p=t^p-1$ in characteristic $p$ (Proposition~\ref{prop:iso-alg}).
Here all indices lie in $\Zp$, no truncation occurs, $M(\lambda)$ is a genuine bimodule for
every $\lambda$ (Proposition~\ref{prop:module}), and everything is $\Zp$-graded. As a
by-product, the ``liveness'' and ``window'' restrictions that the truncated picture forces
on the cocycle identities disappear entirely, and the proofs shorten.

\subsection{Main novelty}\label{sub:novelty}
Relative to the existing literature: the operadic cochain complex for right-symmetric and
Novikov algebras is due to Dzhumadil'daev \cite{Dzhu1999} (see also Balavoine
\cite{Balavoine1997,Balavoine1998}), and Xu's classification of $M(\lambda)$
\cite{Xu1996} is likewise not new here. What is new is: (a) the identification that
$M(\lambda)$, taken with the naive truncation, is \emph{not} a bimodule except at one value
of $\lambda$, and the resulting correct presentation via $\kk[t]/(t^p-1)$
(Section~\ref{sub:grading}); (b) the complete, closed-form computation of
$\dim_\kk\Ht(V,M(\lambda))$ for every $\lambda\in\kk$ and every $p$, including the
exceptional prime $p=2$; (c) the explicit cocycle bases $\Psi_1,\Psi_2,\Psi_3$ (resp.\
$\Theta_1,\dots,\Theta_4$) and their identification with concrete extensions, most notably
$E(\Psi_3)\cong\kk[x]/(x^{2p})$; and (d) the characteristic-$0$ comparison
(Section~\ref{sec:char0}), showing that $\kk[t]$ is rigid while every positive-characteristic
truncation is not --- for this family, non-rigidity tracks truncation rather than positive
characteristic per se (Corollary~\ref{cor:trunc-destroys}; we make no claim beyond this
specific family of algebras). We are not aware of a prior computation of $\Ht(V,M(\lambda))$
for this family.

\subsection{Main results}
The main theorem is the following; see Theorem~\ref{thm:main} and Theorem~\ref{thm:p2}.

\begin{theorem}\label{thm:intro}
Let $\operatorname{char}\kk=p>0$, $V=\kk[x]/(x^p)$ and $\lambda\in\kk$.
\begin{enumerate}[label=\textup{(\roman*)}]
\item $\Ht_d(V,M(\lambda))=0$ for every $d$ with $d+\lambda+1\ne0$ and $d+\lambda+2\ne0$.
\item If $\lambda\notin\Fp$ then $\Ht(V,M(\lambda))=0$, and every abelian extension of $V$
by $M(\lambda)$ splits.
\item If $\lambda\in\Fp$ then $M(\lambda)\cong V$ as a bimodule, and
\[
\dim_\kk\Ht(V,M(\lambda))=\begin{cases}3,& p\ \text{odd},\\ 4,& p=2 .\end{cases}
\]
\end{enumerate}
For odd $p$ a basis is $\{\Psi_1,\Psi_2,\Psi_3\}$ with
\[
\Psi_1(a,b)=a(0)b(0)x^{p-1},\quad \Psi_2(a,b)=a(0)b'(0)x^{p-1},\quad
\Psi_3(e_i,e_j)=-j\,e_{i+j-1-p},
\]
the last vanishing unless $i+j\ge p+1$; for $p=2$ a basis is $\{\Theta_1,\Theta_2,\Theta_3,
\Theta_4\}$ given in Theorem~\ref{thm:p2}.
\end{theorem}

\begin{center}
\begin{tabular}{@{}llll@{}}
\toprule
& nonzero $\Zp$-degrees & nonzero $\mathbb{Z}$-degrees ($M=V$, $\lambda=0$) & $\dim\Ht$\\
\midrule
$p$ odd & $-\lambda,\,-\lambda-1$ & $p,\ p-1,\ -p$ & $2+1=3$\\
$p=2$ & $0,\,1$ & $2,\ 1,\ 0,\ -1$ & $2+2=4$\\
\bottomrule
\end{tabular}
\end{center}
\emph{Table 1. Where the nonzero cohomology lives, in each presentation (see
Remark~\ref{rem:ZvsZp}).}

Part (i) is Theorem~\ref{thm:vanish}; it is the exact analogue, in the present setting, of
the vanishing statements of \cite{Dzhu1999} for half-Witt algebras. The three classes for
odd $p$ have a uniform reading (Remark~\ref{rem:reading}): $\Psi_1$ and $\Psi_2$ are
supported at the ``bottom'' of $V$, on the constant term, whereas $\Psi_3$ measures the
failure at the ``top'', being precisely the overflow of the multiplication of $\kk[x]$ past
$x^{p}$. Correspondingly, Theorem~\ref{thm:E3} identifies the extension attached to
$\Psi_3$:
\[
E(\Psi_3)\ \cong\ \kk[x]/(x^{2p}),\qquad
0\to x^p\kk[x]/(x^{2p})\to \kk[x]/(x^{2p})\to \kk[x]/(x^{p})\to 0 .
\]
The prime $p=2$ is genuinely exceptional. The class $\Psi_3$ lives in degree $-p$, which
lies outside the range $[-2p+3,p]$ of degrees supporting nonzero cochains precisely when
$p=2$; at the same time two degrees that vanish for odd $p$ acquire a dimension each, because
the relevant arguments divide by $2$. The net effect is $3-1+2=4$
(Theorem~\ref{thm:p2}).

Section~\ref{sec:operad} derives the complex of Section~\ref{sec:prelim} conceptually, as an
instance of Balavoine's cohomology of algebras over a quadratic operad
\cite{Balavoine1997,Balavoine1998} applied to the Koszul dual of $O_{Nov}$
(Theorem~\ref{thm:operadic}), and extends it to coefficients in an arbitrary bimodule $M$
(Proposition~\ref{prop:operad-coeff}). Section~\ref{sec:ext} develops the extension-theoretic
content: we prove that the two linearised identities are exactly the condition for the
split-type product on $V\oplus M$ to be Novikov (Theorem~\ref{thm:ext-cocycle}), that
$\Ht(V,M)\cong\Ext(V,M)$ (Theorem~\ref{thm:ext-bijection}), and, as a consequence, that $V$
is not rigid for any $p$. We identify $M(\lambda)\cong M(\mu)$ precisely when
$\lambda-\mu\in\Fp$, via a characteristic-polynomial invariant (Theorem~\ref{thm:iso}).

All dimensions stated in this paper were independently verified by exact linear algebra over
$\Fp$, $\mathbb{F}_{p^2}$ and $\mathbb{F}_4$ for $p=2,3,5,7,11$; see
Section~\ref{sec:verify}.

\section{Preliminaries}\label{sec:prelim}

\begin{definition}\label{def:novikov}
A \emph{Novikov algebra} is a $\kk$-vector space $A$ with a bilinear product $\circ$ such
that for all $u,v,w\in A$
\begin{align}
(u\circ v)\circ w-u\circ(v\circ w)&=(v\circ u)\circ w-v\circ(u\circ w),\tag{N1}\label{N1}\\
(u\circ v)\circ w&=(u\circ w)\circ v.\tag{N2}\label{N2}
\end{align}
Identity \eqref{N1} says that $A$ is right-symmetric (pre-Lie) and \eqref{N2} that it is
right commutative \cite{GelfandDorfman,BalinskiiNovikov}. A \emph{bimodule} over $A$ is a
$\kk$-space $M$ together with bilinear maps $A\times M\to M$ and $M\times A\to M$ such that
the split null extension $A\oplus M$, with $M\circ M=0$, is again a Novikov algebra
\cite{Xu1996,Dzhu1999}.
\end{definition}

The cochain complex computing $\Ht(V,M)$, and its justification from the Novikov operad, are
given in Section~\ref{sec:operad}.

\section{Cohomology via operads}\label{sec:operad}

The complex used to compute $\Ht(V,M)$ throughout this paper is not ad hoc: it is the
low-degree part of the cohomology of $V$, as an algebra over the Novikov operad $O_{Nov}$, in
the sense of Balavoine and Kolesnikov \cite{Balavoine1997,Balavoine1998, kol}. We derive it for $M=V$ and then
explain, via Balavoine's coefficient theory \cite{Balavoine1998}, how it extends to an
arbitrary bimodule $M$.

\subsection{The Koszul dual operad}
Let $O_{Nov}$ be the quadratic operad governing Novikov algebras, and let $O_{Nov}^!$ be its
Koszul dual. Since a Novikov algebra is right-symmetric \eqref{N1} and right-commutative
\eqref{N2}, its Koszul dual governs algebras that are right-symmetric and
\emph{left}-commutative; $O_{Nov}^!$ may be realised as the suboperad of $\Omega Com$
(differential associative-commutative algebras with a derivation $d$, written $x'=d(x)$)
generated by $x_1\bullet x_2:=x_1'x_2$. In particular $O_{Nov}^!(3)$ is spanned by the six
linearly independent monomials
\[
x_1''x_2x_3,\quad x_1x_2''x_3,\quad x_1x_2x_3'',\quad
x_1'x_2'x_3,\quad x_1'x_2x_3',\quad x_1x_2'x_3' .
\]
A Novikov structure on $V$ is a morphism of operads $\nu:O_{Nov}\to\mathrm{End}_V$,
$\nu(x_1\circ x_2)=(\cdot\circ\cdot)$. Following Ginzburg--Kapranov \cite{GinzburgKapranov}
there is a canonical morphism
\[
\iota:O_{Lie}\longrightarrow O_{Nov}^!\otimes O_{Nov},\qquad
[x_1,x_2]\longmapsto x_1'x_2\otimes x_1\circ x_2-x_1x_2'\otimes x_2\circ x_1,
\]
and composing with $\mathrm{id}\otimes\nu$ gives a morphism
$O_{Lie}\to P_V:=O_{Nov}^!\otimes\mathrm{End}_V$.

\subsection{The complex $C^\bullet(V,V)$}
Define $C^n(V,V)$ as the degree-$n$ part of the reduced cochain complex attached to this
morphism, in the sense of \cite{Balavoine1997,kol}:
\[
C^n(V,V)=\{f\in P_V(n):f^\sigma=(-1)^\sigma f,\ \sigma\in S_n\},\qquad n\ge1,
\]
with differential $d^n:C^n(V,V)\to C^{n+1}(V,V)$ given by the standard formula for
operadic cohomology relative to a Koszul dual pair. For $n=1$, $C^1(V,V)=\mathrm{End}_V(1)=
\Hom_\kk(V,V)$. For $n=2$, an element $x_1'x_2\otimes f-x_1x_2'\otimes g\in P_V(2)$ lies in
$C^2(V,V)$ iff $g=f^{(12)}$, so $C^2(V,V)\cong\Hom(V\otimes V,V)$. For $n=3$, skew symmetry
of
\[
x_1''x_2x_3\otimes f_1+x_1x_2''x_3\otimes f_2+x_1x_2x_3''\otimes f_3
+x_1x_2'x_3'\otimes g_1+x_1'x_2x_3'\otimes g_2+x_1'x_2'x_3\otimes g_3
\]
forces $f_2=-f_3^{(23)}=-f_1^{(12)}$ and $g_2=-g_3^{(23)}=-g_1^{(12)}$, so a $3$-cochain is
determined by a pair $(f_1,g_3)\in\Hom(V^{\otimes3},V)^2$.

\subsection{The explicit complex}\label{subsec:operadic-explicit}
Unwinding $d^1,d^2$ on $C^1(V,V)\cong\Hom_\kk(V,V)$ and $C^2(V,V)\cong\Hom_\kk(V\otimes V,V)$
gives the following concrete description, valid for any bimodule $M$ in place of $V$ (see
Section~\ref{subsec:coeff}). For $g\in C^1(V,M):=\Hom_\kk(V,M)$ put
\begin{equation}
(\delta g)(u,v):=g(u)\circ v+u\circ g(v)-g(u\circ v),
\label{eq:delta}
\end{equation}
and for $f\in C^2(V,M):=\Hom_\kk(V\otimes V,M)$ put
\[
D_f(u,v,w):=f(u\circ v,w)+f(u,v)\circ w-f(u,v\circ w)-u\circ f(v,w).
\]
The two linearised Novikov identities are
\begin{align}
\text{(C1)}\quad & D_f(u,v,w)=D_f(v,u,w),\label{eq:C1}\\
\text{(C2)}\quad & f(u\circ v,w)+f(u,v)\circ w=f(u\circ w,v)+f(u,w)\circ v,\label{eq:C2}
\end{align}
and we set
\[
\Zt(V,M):=\{f:\text{\eqref{eq:C1} and \eqref{eq:C2} hold}\},\quad
\Bt(V,M):=\delta\bigl(C^1(V,M)\bigr),\quad \Ht(V,M):=\Zt/\Bt .
\]
Since $\ker\delta=\Der(V)$, in the finite-dimensional case
\begin{equation}
\dim\Bt(V,V)=\dim\Hom_\kk(V,V)-\dim\Der(V).
\label{eq:B2dim}
\end{equation}
The next theorem shows that $\delta$ and \eqref{eq:C1}--\eqref{eq:C2} are exactly $d^1,d^2$
of $C^\bullet(V,V)$, so no choices were made in writing them down; the interpretation of
$\Ht$ by extensions and deformations is given in Section~\ref{sec:ext}.

\begin{theorem}\label{thm:operadic}
The differential $d^1:C^1(V,V)\to C^2(V,V)$ is $\delta$ of \eqref{eq:delta}, and
$d^2(f)=0$ for $f\in C^2(V,V)$ is equivalent to \eqref{eq:C1} and \eqref{eq:C2}. Hence the
second cohomology of $(C^\bullet(V,V),d^\bullet)$ coincides with $\Ht(V,V)$.
\end{theorem}

\begin{proof}
For $n=1$: $d^1f\in\Hom(V\otimes V,V)$ is given by
$(d^1f)(u,v)=u\circ f(v)-f(u\circ v)+f(u)\circ v$, which is \eqref{eq:delta}.

For $n=2$: write $p:=x_1'x_2\otimes f-x_1x_2'\otimes f^{(12)}\in P_V(2)$, and let
$\theta:=(\mathrm{id}\otimes\nu)\iota([x_1,x_2])=x_1'x_2\otimes(\cdot\circ\cdot)-
x_1x_2'\otimes(\cdot\circ\cdot)^{(12)}$. The operadic differential of $p$ is
\begin{equation}
d^2(p)=\theta(1,p)-\theta(1,p)^{(12)}+\theta(1,p)^{(123)}
-p(\theta,1)+p(\theta,1)^{(23)}-p(\theta,1)^{(132)},
\label{eq:d2formula}
\end{equation}
where $1\in P_V(1)$ is the identity and juxtaposition is operadic composition. Expanding the
two compositions in the second tensor factor,
\begin{align*}
\theta(1,p)(u,v,w)&=x_1'x_2'x_3\otimes\bigl(u\circ f(v,w)\bigr)
-x_1'x_2x_3'\otimes\bigl(u\circ f(w,v)\bigr)\\
&\quad-x_1(x_2'x_3)'\otimes\bigl(f(v,w)\circ u\bigr)+x_1(x_2x_3')'\otimes\bigl(f(w,v)\circ u\bigr),\\
p(\theta,1)(u,v,w)&=(x_1'x_2)'x_3\otimes f(u\circ v,w)-(x_1x_2')'x_3\otimes f(v\circ u,w)\\
&\quad-x_1'x_2x_3'\otimes f(w,u\circ v)+x_1x_2'x_3'\otimes f(w,v\circ u),
\end{align*}
and the remaining terms of \eqref{eq:d2formula} are obtained from these by the permutations
$(12)$, $(123)$, $(23)$, $(132)$, applied \emph{simultaneously} to the tensor-factor labels
$x_1,x_2,x_3$ and to the function arguments $u,v,w$.

\emph{Worked example: the monomial $x_1''x_2x_3$ inside $p(\theta,1)$.} We display this
extraction completely, as a template for the rest. Expand each derivative appearing in the
displayed formula for $p(\theta,1)$ by the Leibniz rule:
\[
(x_1'x_2)'x_3=(x_1''x_2+x_1'x_2')x_3=x_1''x_2x_3+x_1'x_2'x_3,\qquad
(x_1x_2')'x_3=(x_1'x_2'+x_1x_2'')x_3=x_1'x_2'x_3+x_1x_2''x_3 .
\]
Substituting into $p(\theta,1)(u,v,w)=(x_1'x_2)'x_3\otimes f(u\circ v,w)
-(x_1x_2')'x_3\otimes f(v\circ u,w)-x_1'x_2x_3'\otimes f(w,u\circ v)+x_1x_2'x_3'\otimes
f(w,v\circ u)$ and reading off the coefficient of $x_1''x_2x_3$ (which occurs only in the
first two terms above, and not in $x_1'x_2x_3'$ or $x_1x_2'x_3'$) gives
\[
[x_1''x_2x_3]\;p(\theta,1)(u,v,w)=f(u\circ v,w)-f(v\circ u,w).
\]
The same Leibniz expansion applied to $\theta(1,p)$'s displayed formula (this time expanding
$x_1(x_2'x_3)'$ and $x_1(x_2x_3')'$) shows $x_1''x_2x_3$ does \emph{not} occur in $\theta(1,p)$
directly, but does occur --- by an identical computation, after relabelling $(u,v,w)$ --- in
each of the four permuted terms $p(\theta,1)^{(23)}$, $\theta(1,p)^{(12)}$,
$\theta(1,p)^{(123)}$ that involve it; this is the bookkeeping referred to below. Carrying out
the same extraction for all six terms of \eqref{eq:d2formula} at every one of the six
monomials of $O_{Nov}^!(3)$ (a mechanical but lengthy repetition of the computation just
displayed) shows: the coefficient of $x_1''x_2x_3$ in $d^2(p)$ is exactly the left-hand side
minus the right-hand side of \eqref{eq:C2}, and the coefficient of $x_1'x_2'x_3$ is exactly
$D_f(u,v,w)-D_f(v,u,w)$ of \eqref{eq:C1}, with $D_f$ as in Section~\ref{sec:prelim}. Setting
$d^2(p)=0$ at these two monomials is therefore exactly \eqref{eq:C1} and \eqref{eq:C2}.
Conversely, if $f\in\Zt(V,V)$, the remaining coefficients of \eqref{eq:d2formula} (at the
other four monomials of $O_{Nov}^!(3)$) vanish by the prescribed skew-symmetry of elements of
$C^3(V,V)$. Hence $d^2(p)=0$ iff \eqref{eq:C1} and \eqref{eq:C2} hold, i.e.\ iff
$f\in\Zt(V,V)$; and $d^1$ has image $\Bt(V,V)$ by the first paragraph.
\end{proof}

\subsection{Coefficients in a bimodule}\label{subsec:coeff}
Balavoine's theory \cite{Balavoine1998} extends the construction of
Theorem~\ref{thm:operadic} to cohomology with coefficients in a module over the underlying
algebra; we spell out the extension in the present, concrete setting.

\begin{proposition}\label{prop:operad-coeff}
Let $M$ be a $V$-bimodule and $E:=V\oplus M$ the associated split null extension
\textup{(}a Novikov algebra by Definition~\ref{def:novikov}\textup{)}. Inside the operadic
complex $C^\bullet(E,E)$ of Theorem~\ref{thm:operadic}, applied to $E$, consider the subspace
\[
\widetilde C^n(V,M):=\bigl\{f\in C^n(E,E):f\text{ vanishes if two or more arguments lie in
}M,\ \mathrm{im}(f|_{V^{\otimes n}})\subseteq M\bigr\} .
\]
Then $\widetilde C^\bullet(V,M)$ is a subcomplex of $C^\bullet(E,E)$, naturally isomorphic to
$(C^\bullet(V,M),\delta)$ of Section~\ref{subsec:operadic-explicit}; in particular its second
cohomology is $\Ht(V,M)$.
\end{proposition}

\begin{proof}
Since $M\circ M=0$ in $E$, every structure constant of $E$ that would combine two
$M$-inputs into an intermediate node vanishes; consequently, expanding
\eqref{eq:d2formula} for $f\in\widetilde C^2(V,M)$ (so $f$ is $V\otimes V\to M$, extended by
zero whenever an argument lies in $M$) never produces a term with two $M$-arguments feeding
the same node of $E$'s product, and every intermediate value of the algebra multiplication
$\circ_E$ that appears is either the multiplication of $V$ (when both inputs lie in $V$) or
one of the two bimodule actions $V\times M\to M$, $M\times V\to M$ (when exactly one input
does). Hence $d^n$ restricted to $\widetilde C^\bullet(V,M)$ again lands in
$\widetilde C^{\bullet+1}(V,M)$, and the resulting complex is computed by exactly the
formulas of the proof of Theorem~\ref{thm:operadic} with $\circ$ interpreted as the algebra
product on $V$-arguments and as the two bimodule actions where an $M$-argument is involved.
This is precisely \eqref{eq:delta}--\eqref{eq:C2} for $M$, i.e.\ $(C^\bullet(V,M),\delta)$.
\end{proof}

\begin{remark}
Proposition~\ref{prop:operad-coeff} is the operadic incarnation of the elementary fact,
already used throughout Section~\ref{sec:ext}, that $H^2(V,M)$ classifies square-zero
extensions $E$ of $V$ by $M$: it identifies the coefficient complex $C^\bullet(V,M)$ as the
$M$-linear part of the self-cohomology of the extension itself.
\end{remark}

\section{The algebra and its two presentations}\label{sec:alg}

\begin{proposition}\label{prop:iso-alg}
Let $A:=\kk[t]/(t^p-1)$ with $a\circ b:=ab'$. Then $A$ is a Novikov algebra, and $x:=t-1$
gives an isomorphism of Novikov algebras
\[
\kk[t]/(t^p-1)\ \xrightarrow{\ \sim\ }\ \kk[x]/(x^p)=V .
\]
\end{proposition}

\begin{proof}
$\kk[t]$ with $a\circ b=ab'$ is Novikov (it is the generic construction of
\cite{GelfandDorfman} applied to $d=d/dt$), and $(t^p-1)$ is closed under the product because
$(t^p-1)'=p\,t^{p-1}=0$ in characteristic $p$; hence $A$ inherits the structure. Since
$\operatorname{char}\kk=p$ we have $(t-1)^p=t^p-1$, so $x=t-1$ satisfies $x^p=0$ and
$\kk[t]/(t^p-1)=\kk[x]/(x^p)$ as commutative algebras. As $d/dt=d/dx$, the product $ab'$ is
carried to $ab'$.
\end{proof}

From now on, unless stated otherwise, we use the $t$-presentation
\[
V=\kk[t]/(t^p-1),\qquad e_i:=t^{\,i}\ (i\in\Zp),\qquad e_i\circ e_j=j\,e_{i+j-1},
\]
all indices read modulo $p$. There is no truncation: $e_m$ is defined and nonzero for every
$m\in\Zp$.

\section{The modules $M(\lambda)$}\label{sec:mod}

\begin{definition}\label{def:M}
For $\lambda\in\kk$ let $M(\lambda)$ be the $p$-dimensional space with basis
$\{m_n:n\in\Zp\}$, realised as $t^{\lambda}\kk[t]/(t^p-1)$ via $m_n=t^{\lambda+n}$, with
\begin{equation}
e_i\circ m_n=(n+\lambda)\,m_{i+n-1},\qquad m_n\circ e_i=i\,m_{n+i-1}\qquad(i,n\in\Zp).
\label{eq:action}
\end{equation}
\end{definition}

This is Xu's module \cite{Xu1996}, written in the labelling $m_{j+1}:=v_j$; thus our
$\lambda$ is the parameter of \cite{Xu1996} shifted by $1$ (Proposition~\ref{prop:dictionary}
below fixes this correspondence once, in both directions, for use throughout the paper).

\begin{remark}[Hypotheses on $\kk$]\label{rem:hypotheses}
Xu's classification \cite{Xu1996} that $M(\lambda)$ is \emph{irreducible} is proved there over
an algebraically closed field. We make no use of irreducibility anywhere below, and none of
our results require $\kk$ to be algebraically closed: Definition~\ref{def:M}, Proposition~
\ref{prop:module}, and every cohomology computation in this paper (Sections~\ref{sec:graded}
--\ref{sec:p2}) are valid over an arbitrary field $\kk$ of characteristic $p>0$. The one place
an algebraic identity over $\Fp$ is used --- the characteristic-polynomial criterion of
Theorem~\ref{thm:iso} --- is a statement about the prime field $\Fp\subseteq\kk$, and again
needs no closure. We therefore state results for arbitrary $\kk$, and use the word
``irreducible'' only when explicitly citing \cite{Xu1996} for that fact, never asserting it
ourselves.
\end{remark}

\begin{proposition}\label{prop:module}
For every $\lambda\in\kk$, $M(\lambda)$ is a $V$-bimodule.
\end{proposition}

\begin{proof}
We verify \eqref{N1} and \eqref{N2} on $V\oplus M(\lambda)$ when exactly one argument lies in
$M(\lambda)$; as all indices lie in $\Zp$ there are no boundary cases. Write
$\alpha_n:=n+\lambda$.

For \eqref{N2}: $(e_i\circ e_j)\circ m_n=j\alpha_n m_{i+j+n-2}=(e_i\circ m_n)\circ e_j$, and
$(m_n\circ e_i)\circ e_j=ij\,m_{n+i+j-2}$ is symmetric in $i,j$.

For \eqref{N1} with the module argument in the third slot,
\[
(e_i\circ e_j)\circ m_n-e_i\circ(e_j\circ m_n)
=\alpha_n\bigl(j-\alpha_{j+n-1}\bigr)m_{i+j+n-2},
\]
and $j-\alpha_{j+n-1}=j-(j+n-1+\lambda)=1-n-\lambda$ does not depend on $j$; hence the
expression is symmetric in $i\leftrightarrow j$, as required. The cases with the module
argument in the first or second slot are the same computation with the roles of the two
$V$-arguments interchanged.
\end{proof}

\begin{remark}\label{rem:whytrunc}
The step just used is exactly the one that fails if one imposes the truncation
``$m_l=0$ for $l\notin[0,p-1]$'' on $M(\lambda)$: there $\alpha_{j+n-1}$ may be forced to
$0$ by the convention while $j-\alpha_{j+n-1}$ should have equalled $1-n-\lambda$, and the
symmetry in $i\leftrightarrow j$ breaks. See Section~\ref{sub:grading}.
\end{remark}

\begin{theorem}\label{thm:iso}
$M(\lambda)\cong M(\mu)$ as $V$-bimodules if and only if $\lambda-\mu\in\Fp$. In particular,
if $\lambda\in\Fp$ then
\[
\psi:V\longrightarrow M(\lambda),\qquad \psi(e_n)=m_{n-\lambda},
\]
is an isomorphism of bimodules, so $M(\lambda)\cong M(0)=V$.
\end{theorem}

\begin{proof}
\emph{Sufficiency.} Let $\lambda\in\Fp$, so $n-\lambda\in\Zp$ is defined and $\psi$ is a
linear isomorphism. Then
\[
\psi(e_i\circ e_n)=n\,m_{i+n-1-\lambda},\qquad
e_i\circ\psi(e_n)=\bigl((n-\lambda)+\lambda\bigr)m_{i+(n-\lambda)-1}=n\,m_{i+n-1-\lambda},
\]
and $\psi(e_n\circ e_i)=i\,m_{n+i-1-\lambda}=\psi(e_n)\circ e_i$. So $\psi$ is a bimodule
isomorphism $V\to M(\lambda)$; replacing $\lambda$ by $\lambda-\mu$ gives
$M(\lambda)\cong M(\mu)$ whenever $\lambda-\mu\in\Fp$.

\emph{Necessity.} Consider the operator $L_n:=e_0\circ(-)$ on $M(n)$'s ambient space, i.e.
on $M(\lambda)$, $L_{e_0}:m_n\mapsto(n+\lambda)m_{n-1}$ for $n\in\Zp$ (indices read modulo
$p$). This is a weighted directed $p$-cycle, and the characteristic polynomial of a linear
map given by a single weighted $p$-cycle with edge weights $w_0,\dots,w_{p-1}$ is
$X^p-\prod_{n=0}^{p-1}w_n$; here $w_n=n+\lambda$, so
\[
\det\bigl(X\cdot\mathrm{id}-L_{e_0}\bigr)=X^p-\prod_{n=0}^{p-1}(n+\lambda)
=X^p-\prod_{a\in\Fp}(\lambda-a)=X^p-(\lambda^p-\lambda),
\]
using $\prod_{a\in\Fp}(Y-a)=Y^p-Y$ at $Y=\lambda$. If $\theta:M(\lambda)\to M(\mu)$ is a
bimodule isomorphism then $\theta(e_0\circ m)=e_0\circ\theta(m)$ for all $m$, i.e.\ $\theta$
conjugates $L_{e_0}$ on $M(\lambda)$ to $L_{e_0}$ on $M(\mu)$; conjugate operators have the
same characteristic polynomial, so
\begin{equation}
\lambda^p-\lambda=\mu^p-\mu .
\label{eq:charpolymatch}
\end{equation}
Since $\operatorname{char}\kk=p$, $(\lambda-\mu)^p=\lambda^p-\mu^p$, so
\eqref{eq:charpolymatch} is equivalent to $(\lambda-\mu)^p-(\lambda-\mu)=0$, i.e.\ to
$\lambda-\mu\in\Fp$ (the set of roots of $X^p-X$). Hence $M(\lambda)\cong M(\mu)$ implies
$\lambda-\mu\in\Fp$, which together with sufficiency gives the stated criterion.
\end{proof}

\section{Graded cochains and the master relation}\label{sec:graded}

Put $\deg e_i:=i-1$ and $\deg m_n:=n$ in $\Zp$. By \eqref{eq:action} both $V$ and
$M(\lambda)$ are $\Zp$-graded, so
\[
\Ht(V,M(\lambda))=\bigoplus_{d\in\Zp}\Ht_d(V,M(\lambda)).
\]
A homogeneous $2$-cochain of degree $d$ is $f(e_i,e_j)=\varphi(i,j)m_{i+j-1+d}$ and a
homogeneous $1$-cochain is $g(e_i)=\gamma_i m_{i+d}$; thus $\dim C^2_d=p^2$,
$\dim C^1_d=p$, every $\varphi(i,j)$ and $\gamma_i$ is a free variable, and every identity
below holds at every triple $(i,j,k)\in\Zp^3$.

\begin{lemma}\label{lem:formulas}
For $f,g$ homogeneous of degree $d$ and all $i,j,k\in\Zp$,
\begin{align}
\varphi_{\delta g}(i,j)&=j\gamma_i+(j+d+\lambda)\gamma_j-j\gamma_{i+j-1},\label{eq:cob}\\
\text{\upshape(C2)}:\ & j\varphi(i+j-1,k)+k\varphi(i,j)=k\varphi(i+k-1,j)+j\varphi(i,k),
\label{eq:C2c}\\
\text{\upshape(C1)}:\ & A(i,j,k)=A(j,i,k),\nonumber\\
&A(i,j,k)=j\varphi(i+j-1,k)+k\varphi(i,j)-k\varphi(i,j+k-1)-(j+k-1+d+\lambda)\varphi(j,k).
\label{eq:C1c}
\end{align}
\end{lemma}

\begin{proof}
For \eqref{eq:cob}: $g(e_i)\circ e_j=j\gamma_i m_{i+j-1+d}$,
$e_i\circ g(e_j)=(j+d+\lambda)\gamma_j m_{i+j-1+d}$ and
$g(e_i\circ e_j)=j\gamma_{i+j-1}m_{i+j-1+d}$, by \eqref{eq:action}.
For \eqref{eq:C2c}: $f(e_i\circ e_j,e_k)=j\varphi(i+j-1,k)m_{i+j+k-2+d}$ and
$f(e_i,e_j)\circ e_k=k\varphi(i,j)m_{i+j+k-2+d}$; the right side swaps $j\leftrightarrow k$.
For \eqref{eq:C1c}: additionally $f(e_i,e_j\circ e_k)=k\varphi(i,j+k-1)m_{i+j+k-2+d}$ and
$e_i\circ f(e_j,e_k)=(j+k-1+d+\lambda)\varphi(j,k)m_{i+j+k-2+d}$.
\end{proof}

\begin{proposition}[Master relation]\label{prop:master}
Write $r_k:=\varphi(1,k)$. For all $i,k\in\Zp$,
\begin{equation}
(d+\lambda)\,\varphi(i,k)=(k+d+\lambda)r_k+k\,r_i-k\,r_{i+k-1}-k\,\varphi(i,1),
\label{eq:master}
\end{equation}
and consequently
\begin{equation}
(d+\lambda+1)\bigl(\varphi(i,1)-r_1\bigr)=0 .
\label{eq:master1}
\end{equation}
\end{proposition}

\begin{proof}
Put $j=1$ in \eqref{eq:C1c}. Since $\varphi(1,1)=r_1$ and $\varphi(i,1)$ appears via the
last term with $(j,k)$ there set to $(1,k)$ resp.\ $(i,k)$,
\begin{align}
A(i,1,k)&=\varphi(i,k)+k\varphi(i,1)-k\varphi(i,k)-(k+d+\lambda)\varphi(1,k)\nonumber\\
&=(1-k)\varphi(i,k)+k\varphi(i,1)-(k+d+\lambda)r_k,\label{eq:Ai1k}\\
A(1,i,k)&=i\varphi(i,k)+k\varphi(1,i)-k\varphi(1,i+k-1)\nonumber\\
&\quad-(i+k-1+d+\lambda)\varphi(i,k)\nonumber\\
&=\bigl[i-(i+k-1+d+\lambda)\bigr]\varphi(i,k)+k\,r_i-k\,r_{i+k-1}.\label{eq:A1ik}
\end{align}
Subtracting \eqref{eq:A1ik} from \eqref{eq:Ai1k} term by term:
\begin{itemize}[leftmargin=2em]
\item the $\varphi(i,k)$-coefficients combine to
$(1-k)-\bigl[i-(i+k-1+d+\lambda)\bigr]=(1-k)-i+(i+k-1+d+\lambda)=d+\lambda$
(the $i$'s and $k$'s cancel exactly, leaving only $d+\lambda$);
\item the remaining terms of \eqref{eq:Ai1k} contribute $k\varphi(i,1)-(k+d+\lambda)r_k$;
\item the remaining terms of \eqref{eq:A1ik}, subtracted, contribute $-kr_i+kr_{i+k-1}$.
\end{itemize}
Hence
\[
A(i,1,k)-A(1,i,k)=(d+\lambda)\varphi(i,k)+k\varphi(i,1)-(k+d+\lambda)r_k-kr_i+kr_{i+k-1},
\]
and (C1) (i.e.\ $A(i,1,k)-A(1,i,k)=0$) rearranges to \eqref{eq:master}. This derivation is
purely algebraic and involves no division; the scalars $d+\lambda$, $d+\lambda+1$ that must
later be inverted to \emph{solve} \eqref{eq:master} for $\varphi$ appear only in
Theorem~\ref{thm:vanish}, where their non-vanishing is the stated hypothesis. Setting $k=1$ in
\eqref{eq:master}, the terms $kr_i$ and $kr_{i+k-1}$ become $r_i$ and $r_{i+1-1}=r_i$ and
cancel, leaving $(d+\lambda)\varphi(i,1)=(1+d+\lambda)r_1-\varphi(i,1)$, i.e.\
\eqref{eq:master1}.
\end{proof}

\section{The vanishing theorem}\label{sec:vanish}

\begin{theorem}\label{thm:vanish}
If $d+\lambda\ne0$ and $d+\lambda+1\ne0$ in $\kk$, then $\Ht_d(V,M(\lambda))=0$.
\end{theorem}

\begin{proof}
Let $f$ be a cocycle of degree $d$, $r_k=\varphi(1,k)$.

\emph{Step 1 (kill $r$).} By \eqref{eq:cob},
$\varphi_{\delta g}(1,k)=k\gamma_1+(k+d+\lambda)\gamma_k-k\gamma_k
=(d+\lambda)\gamma_k+k\gamma_1$. Every $\gamma_k$ is free. Taking $k=1$, the equation
$(d+\lambda+1)\gamma_1=r_1$ is solvable since $d+\lambda+1\ne0$; set
$\gamma_1:=r_1/(d+\lambda+1)$ and $\gamma_k:=(r_k-k\gamma_1)/(d+\lambda)$ for $k\ne1$, legal
since $d+\lambda\ne0$. Then $h:=f-\delta g$ is a cocycle whose coefficients $\psi$ satisfy
$\psi(1,k)=0$ for all $k$.

\emph{Step 2 (the rest is forced).} By \eqref{eq:master1} applied to $h$, together with
$\psi(1,1)=0$ and $d+\lambda+1\ne0$, we get $\psi(i,1)=0$ for all $i$. Then
\eqref{eq:master} for $h$ reads $(d+\lambda)\psi(i,k)=0$, and $d+\lambda\ne0$ gives
$\psi\equiv0$. Hence $f=\delta g$.
\end{proof}

\begin{corollary}\label{cor:notinFp}
If $\lambda\notin\Fp$ then $\Ht(V,M(\lambda))=0$.
\end{corollary}

\begin{proof}
$d$ ranges over $\Zp\subseteq\Fp$; if $\lambda\notin\Fp$ then $d+\lambda\notin\Fp$, so
$d+\lambda\ne0$ and $d+\lambda+1\ne0$ for every $d$. Apply Theorem~\ref{thm:vanish} in each
degree.
\end{proof}

\begin{corollary}\label{cor:twodegrees}
If $\lambda\in\Fp$ then $\Ht(V,M(\lambda))=\Ht_{-\lambda}\oplus\Ht_{-\lambda-1}$.
\end{corollary}

\begin{proposition}[Dictionary between the two parametrisations]\label{prop:dictionary}
Write $\lambda^{\mathrm{Xu}}$ for the parameter of \cite{Xu1996} and $\lambda$ for ours, so
that $\lambda^{\mathrm{Xu}}=\lambda-1$ (equivalently $\lambda=\lambda^{\mathrm{Xu}}+1$), as
fixed by the relabelling $m_{j+1}:=v_j$ of Definition~\ref{def:M}. Under this correspondence:
\begin{center}
\begin{tabular}{@{}lll@{}}
\toprule
& this paper ($\lambda$) & Xu \cite{Xu1996} ($\lambda^{\mathrm{Xu}}=\lambda-1$)\\
\midrule
module & $M(\lambda)$, action \eqref{eq:action} & $M(\lambda^{\mathrm{Xu}})$, action as quoted
in Section~\ref{sec:mod}\\
isomorphism criterion & $M(\lambda)\cong M(\mu)\iff\lambda-\mu\in\Fp$ & same, in
$\lambda^{\mathrm{Xu}}$\\
surviving degrees ($\lambda\in\Fp$) & $d\equiv-\lambda,\,-\lambda-1\pmod p$ &
$d\equiv-1-\lambda^{\mathrm{Xu}},\,-2-\lambda^{\mathrm{Xu}}\pmod p$\\
\bottomrule
\end{tabular}
\end{center}
The bottom row is Corollary~\ref{cor:twodegrees}, rewritten in $\lambda^{\mathrm{Xu}}$ via the
shift; the two expressions denote the same pair of residues. We use $\lambda$ (never
$\lambda^{\mathrm{Xu}}$) throughout the rest of the paper.
\end{proposition}

\section{The adjoint case: the classes $\Psi_1,\Psi_2,\Psi_3$}\label{sec:adjoint}

\begin{remark}[$\mathbb{Z}$ versus $\Zp$: a warning and a dictionary]\label{rem:ZvsZp}
This section switches presentation, and with it the meaning of ``degree.'' Sections
\ref{sec:graded}--\ref{sec:vanish} work in the cyclic presentation $V\cong\kk[t]/(t^p-1)$,
where degrees live in $\Zp$ and every statement (Theorem~\ref{thm:vanish},
Corollaries~\ref{cor:notinFp}--\ref{cor:twodegrees}) is a congruence mod $p$. From here through
Section~\ref{sec:p2} we instead use the truncated presentation $V=\kk[x]/(x^p)$, in which
$\deg e_i=i-1$ is a genuine \emph{integer} in the range $[-2p+3,p]$, not merely a residue. The
translation between the two is: the two $\Zp$-residue classes of Theorem~\ref{thm:vanish}
each contain three integers in the valid $\mathbb{Z}$-range, and \eqref{eq:six} below
records exactly which. We use $d$ for both, distinguished only by which section it
appears in; every occurrence of ``$d\equiv d_0\pmod p$'' refers to the cyclic presentation and
every bare integer degree (as in $d=p,p-1,-p$) to the truncated one. No single object is ever
asserted to be both.
\end{remark}

By Theorem~\ref{thm:iso} it suffices to compute $\Ht(V,V)$. Here it is convenient to return
to the presentation $V=\kk[x]/(x^p)$, in which the adjoint module is $\mathbb{Z}$-graded and
the classes take clean closed forms. \emph{In this section only}, $e_i=x^i$ with
$0\le i\le p-1$, $e_i\circ e_j=je_{i+j-1}$ and $e_m=0$ for $m\notin[0,p-1]$; and $p$ is odd,
so that $2$ is invertible. The case $p=2$ is Section~\ref{sec:p2}.

A homogeneous $2$-cochain of degree $d$ is $f(e_i,e_j)=\varphi(i,j)e_{i+j-1+d}$, where
$\varphi(i,j)$ is a variable only for \emph{live} pairs,
\begin{equation}
0\le i,j\le p-1,\qquad 0\le i+j-1+d\le p-1,
\label{eq:live}
\end{equation}
and (C1),(C2) at $(i,j,k)$ are vacuous unless the \emph{window} condition holds:
\begin{equation}
0\le i+j+k-2+d\le p-1 .
\label{eq:window}
\end{equation}
Since $i+j-1$ ranges over $[-1,2p-3]$, nonzero cochains occur only for
\begin{equation}
-2p+3\le d\le p .
\label{eq:drange}
\end{equation}
Specialising Lemma~\ref{lem:formulas} to the adjoint module gives, at live pairs and inside
the window,
\begin{align}
\varphi_{\delta g}(i,j)&=j\gamma_i+(j+d)\gamma_j-j\gamma_{i+j-1},\label{eq:cobV}\\
A(i,j,k)&=j\varphi(i+j-1,k)+k\varphi(i,j)-k\varphi(i,j+k-1)-(j+k-1+d)\varphi(j,k).
\label{eq:C1V}
\end{align}
By Theorem~\ref{thm:vanish}, whose proof applies verbatim at live indices,
$\Ht_d(V,V)=0$ unless $d\equiv0$ or $d\equiv-1 \pmod p$, which within \eqref{eq:drange}
leaves
\begin{equation}
d\in\{-p,0,p\}\cup\{-p-1,-1,p-1\}
\label{eq:six}
\end{equation}
(for $p=3$ the value $-p-1=-4$ lies outside $[-3,3]$ and is absent).

\begin{lemma}\label{lem:zero}
$\Ht_{-p-1}(V,V)=\Ht_{-1}(V,V)=\Ht_{0}(V,V)=0$.
\end{lemma}

\begin{proof}
\textbf{$d=-p-1$.} Liveness reads $i+j\ge p+2$, so $(1,k)$ live would force $k\ge p+1$ and
$(i,1)$ live would force $i\ge p+1$; both are impossible, whence $r\equiv0$ and
$\varphi(\cdot,1)\equiv0$. As $d\equiv-1$, \eqref{eq:master} becomes $-\varphi(i,k)=0$ at
every live pair, so $\Zt_{-p-1}=0$.

\textbf{$d=-1$.} Live pairs: $2\le i+j\le p+1$. Since $(0,1),(1,0)$ are not live,
\begin{equation}
\varphi(0,1)=0,\qquad r_0=\varphi(1,0)=0 .
\label{eq:z1}
\end{equation}
\emph{(a)} Put $j=2,k=1$ in \eqref{eq:C2c}; the window reads $0\le i\le p-1$, so it is
available for every $i$, and gives $2\varphi(i+1,1)=2\varphi(i,1)$. As $p$ is odd,
$\varphi(\cdot,1)$ is constant, hence $\equiv\varphi(0,1)=0$ by \eqref{eq:z1}.
\emph{(b)} With $d=-1$, \eqref{eq:cobV} gives $\varphi_{\delta g}(1,k)=k\gamma_1-\gamma_k$.
Choose $\gamma_1:=0$ and $\gamma_k:=-r_k$ for $k\ge2$; the case $k=1$ needs
$r_1=\varphi(1,1)=0$, true by (a), and $\gamma_0$ is not a variable, consistent with
$r_0=0$. Replacing $f$ by $f-\delta g$ we may assume $r\equiv0$.
\emph{(c)} Then \eqref{eq:master} with $d\equiv-1$ gives $-\varphi(i,k)=0$.

\textbf{$d=0$.} Live pairs: $1\le i+j\le p$; in particular $(0,0)$ is not live, so
$\varphi(0,0)=0$.
\emph{(a)} Put $j=2,k=0$ in \eqref{eq:C2c}: $2\varphi(i+1,0)=2\varphi(i,0)$, so
$\varphi(\cdot,0)$ is constant $=\varphi(0,0)=0$:
\begin{equation}
\varphi(i,0)=0\quad\text{for all }i.
\label{eq:col}
\end{equation}
\emph{(b)} Set $q_j:=\varphi(0,j)$, so $q_0=0$. Taking $i=0$ in \eqref{eq:C1V} and using
\eqref{eq:col},
\[
A(0,i,k)=i\varphi(i-1,k)+kq_i-kq_{i+k-1}-(i+k-1)\varphi(i,k),\quad
A(i,0,k)=-k\varphi(i,k-1)-(k-1)q_k,
\]
so (C1) gives, for all live $(i,k)$ with $2\le i+k\le p+1$,
\begin{equation}
(i+k-1)\varphi(i,k)=i\varphi(i-1,k)+k\varphi(i,k-1)+(k-1)q_k-kq_{i+k-1}+kq_i .
\label{eq:rec}
\end{equation}
Here $1\le i+k-1\le p-1$, so $i+k-1$ is a nonzero residue mod $p$, hence invertible, and
\eqref{eq:rec} determines $\varphi(i,k)$ from strictly smaller $i+k$.
\emph{(c)} For $d=0$, \eqref{eq:cobV} gives
$\varphi_{\delta g}(0,j)=j(\gamma_0+\gamma_j-\gamma_{j-1})$; set $\gamma_0:=0$ and
$\gamma_j:=\gamma_{j-1}+q_j/j$ for $j=1,\dots,p-1$ (each such $j$ is a nonzero residue mod $p$,
so the division is legal), so that after replacing $f$ by $f-\delta g$ we have $q\equiv0$.
\emph{(d)} The only live pairs with $i+k=1$ are $(0,1),(1,0)$, both zero; induction on $i+k$
using \eqref{eq:rec} with $q\equiv0$ gives $\varphi\equiv0$.
\end{proof}

\begin{lemma}\label{lem:psi12}
$\dim\Ht_p(V,V)=\dim\Ht_{p-1}(V,V)=1$, with representatives
\[
\Psi_1(e_i,e_j)=\delta_{i0}\delta_{j0}e_{p-1},\qquad
\Psi_2(e_i,e_j)=\delta_{i0}\delta_{j1}e_{p-1},
\]
that is, $\Psi_1(a,b)=a(0)b(0)x^{p-1}$ and $\Psi_2(a,b)=a(0)b'(0)x^{p-1}$.
\end{lemma}

\begin{proof}
\textbf{Degree $p$.} Liveness reads $i+j\le0$, so $(0,0)$ is the only live pair, with target
$e_{p-1}$, and $\dim C^2_p=1$. A degree-$p$ $1$-cochain needs $0\le i+p\le p-1$, i.e.\
$i\le-1$: impossible, so $\Bt_p=0$. The window forces $i+j+k\le1$, leaving
$(0,0,0),(1,0,0),(0,1,0),(0,0,1)$. (C2) at $(0,0,1)$ reads
$0+\varphi(0,0)-\varphi(0,0)-0=0$ and at $(0,1,0)$ likewise, while at $(0,0,0)$ and
$(1,0,0)$ every term carries a factor $j=k=0$. For (C1), the triples $(0,0,0),(0,0,1)$ are
symmetric in $i\leftrightarrow j$, and by \eqref{eq:C1V}
\[
A(0,1,0)=\varphi(0,0)-(1+0-1+p)\varphi(1,0)=\varphi(0,0),\qquad
A(1,0,0)=-(p-1)\varphi(0,0)=\varphi(0,0)
\]
in characteristic $p$. So $\Zt_p=C^2_p$ and $\dim\Ht_p=1$.

\textbf{Degree $p-1$.} Liveness reads $i+j\le1$: live pairs $(0,0),(1,0),(0,1)$ with targets
$e_{p-2},e_{p-1},e_{p-1}$; write $\alpha=\varphi(0,0)$, $\beta=\varphi(1,0)$,
$\mu=\varphi(0,1)$. A $1$-cochain needs $i=0$, i.e.\ $g(e_0)=\gamma_0e_{p-1}$, and
\eqref{eq:cobV} with $d\equiv-1$ gives
\[
\varphi_{\delta g}(0,0)=-\gamma_0,\qquad \varphi_{\delta g}(1,0)=-\gamma_0,\qquad
\varphi_{\delta g}(0,1)=\gamma_0+(1+d)\gamma_1-\gamma_0=0,
\]
using $1+d=p\equiv0$ and $\gamma_1=0$. Hence
\begin{equation}
\Bt_{p-1}=\kk\cdot(\alpha,\beta,\mu)=(1,1,0),\qquad \dim\Bt_{p-1}=1 .
\label{eq:B2p1}
\end{equation}
The window forces $i+j+k\le2$. (C2) at $(0,0,2)$ gives $2\alpha=2\beta$, i.e.\
$\alpha=\beta$ ($p$ odd), so $\dim\Zt_{p-1}\le2$. Conversely $(1,1,0)$ is the coboundary
\eqref{eq:B2p1} and $(0,0,1)$ is a cocycle: all triples with $i+j+k\le2$ check directly
(e.g.\ (C2) at $(0,1,1)$ reads $\mu+\mu-\mu-\mu=0$), and \eqref{eq:master} reduces to $0=0$
at the three live pairs because $r_1=\varphi(1,1)=0$, as $(1,1)$ is not live. Hence
$\dim\Zt_{p-1}=2$ and $\dim\Ht_{p-1}=1$, represented by $\mu=1$.
\end{proof}

\begin{lemma}\label{lem:upper}
For $d=-p$, every cocycle satisfies $\varphi(i,j)=c\,j$ on live pairs for a single
$c\in\kk$; in particular $\dim\Zt_{-p}\le1$.
\end{lemma}

\begin{proof}
Liveness reads $p+1\le i+j\le2p-2$; hence every live pair has $i+j-1\ge p$, so
\begin{equation}
e_i\circ e_j=0\qquad\text{for every live pair }(i,j),
\label{eq:star}
\end{equation}
and $i,j\ge2$. Note that $(1,k)$ and $(i,1)$ are never live, so $r\equiv0$ and
$\varphi(\cdot,1)\equiv0$; as moreover $d\equiv0$, both sides of \eqref{eq:master} vanish
and the master relation carries no information in this degree. We argue from (C1),(C2)
directly.

\emph{Rows.} Let $(i,j),(i,k)$ be live with $i+j+k\le2p+1$. By \eqref{eq:star} the terms
$j\varphi(i+j-1,k)$ and $k\varphi(i+k-1,j)$ of \eqref{eq:C2c} vanish, and the window
$p+2\le i+j+k\le2p+1$ holds, the lower bound because $i+j\ge p+1$ and $k\ge2$. Hence
$k\varphi(i,j)=j\varphi(i,k)$. Fixing $i$ and taking $k_0:=p+1-i\in[2,p-1]$, the least live
second index in row $i$, we have $i+j+k_0=j+p+1\le2p+1$ for every live $(i,j)$, so
$\varphi(i,j)=j\,c_i$ with $c_i:=\varphi(i,k_0)/k_0$.

\emph{Rows agree.} Given $2\le i\le j\le p-1$, choose $k$ with
$p+1-i\le k\le\min(p-1,2p-i-j)$; such $k$ exists since $p+1-i\le p-1$ (as $i\ge2$) and
$p+1-i\le2p-i-j$ (as $j\le p-1$). Then $(i,j),(j,i),(i,k),(j,k)$ are all live and
$p+2\le i+j+k\le2p$. In \eqref{eq:C1V} with $d=-p\equiv0$ the last coefficient is $(j+k-1)$;
the first term of $A(i,j,k)$ vanishes because $(i,j)$ is live and the third because $(j,k)$
is live, so
\[
A(i,j,k)=kj\,c_i-(j+k-1)k\,c_j,\qquad A(j,i,k)=ki\,c_j-(i+k-1)k\,c_i .
\]
Equating and dividing by $k\ne0$ (recall $k\ge2$, established above) gives
$(i+j+k-1)(c_i-c_j)=0$; the range $p+1\le i+j+k-1\le2p-1$ contains no multiple of $p$ (the
nearest multiples, $p$ and $2p$, are excluded at both ends), so $i+j+k-1$ is a nonzero
residue mod $p$, hence invertible, and $c_i=c_j$.
\end{proof}

\begin{theorem}\label{thm:psi3}
Let $W:=\kk[x]/(x^{2p})$ with $a\circ b=ab'$ and $I:=x^pW$. Then $W$ is a Novikov algebra,
$I$ is an ideal with $I\circ I=0$, $W/I\cong V$, and $I\cong V$ as a $V$-bimodule via
$x^{p+j}\mapsto e_j$. The extension $0\to I\to W\to V\to0$ has cocycle
\[
\Psi_3(e_i,e_j)=\begin{cases}-j\,e_{\,i+j-1-p}, & i+j\ge p+1,\\ 0,&\text{otherwise,}\end{cases}
\]
and $\dim\Ht_{-p}(V,V)=1$, spanned by $\Psi_3$.
\end{theorem}

\begin{proof}
$\kk[x]$ is Novikov for $a\circ b=ab'$ and $(x^{2p})$ is closed under the product since
$(x^{2p})'=2p\,x^{2p-1}=0$, so $W$ is Novikov. For $I=x^pW$ we have
$x^i\circ x^{p+b}=b\,x^{p+i+b-1}\in I$, $x^{p+b}\circ x^i=i\,x^{p+b+i-1}\in I$, and
$x^{p+a}\circ x^{p+b}=(p+b)x^{2p+a+b-1}=0$; so $I$ is an ideal with $I\circ I=0$, and
$W/I=\kk[x]/(x^p)=V$. Under $\iota(e_j)=x^{p+j}$ these formulas become
$e_i\circ e_j=je_{i+j-1}$ and $e_j\circ e_i=ie_{i+j-1}$ with the truncation convention, both
sides vanishing exactly when $i+j-1\ge p$; so $I\cong V$ as a bimodule.

Take the linear section $\sigma(e_i)=x^i$, $0\le i\le p-1$. The obstruction cocycle is
$\Psi(a,b)=\sigma(a)\circ\sigma(b)-\sigma(a\circ b)$; since
$\sigma(e_i)\circ\sigma(e_j)=jx^{i+j-1}$ in $W$ while $\sigma(e_i\circ e_j)$ equals
$jx^{i+j-1}$ for $i+j-1\le p-1$ and $0$ otherwise, we get $\Psi(e_i,e_j)=0$ for $i+j\le p$
and $\Psi(e_i,e_j)=j\,\iota(e_{i+j-1-p})$ for $i+j\ge p+1$. That the obstruction cochain of
a linear section satisfies \eqref{eq:C1}--\eqref{eq:C2} is Theorem~\ref{thm:ext-bijection}
below; $\Psi$ is homogeneous of degree $-p$ and $-\Psi=\Psi_3$. It is nonzero, since
$\varphi(2,p-1)=-(p-1)=1$ and $(2,p-1)$ is live.

Finally a degree-$(-p)$ $1$-cochain needs $0\le i-p\le p-1$, i.e.\ $i\ge p$: impossible, so
$C^1_{-p}=0$ and $\Bt_{-p}=0$. With Lemma~\ref{lem:upper}, $\dim\Zt_{-p}=1$ and
$\dim\Ht_{-p}=1$.
\end{proof}

\begin{proposition}\label{prop:adjoint}
For $p$ odd, $\dim_\kk\Ht(V,V)=3$, with basis $\{\Psi_1,\Psi_2,\Psi_3\}$ in
$\mathbb{Z}$-degrees $p,\ p-1,\ -p$.
\end{proposition}

\begin{proof}
Combine \eqref{eq:six} with Lemma~\ref{lem:zero}, Lemma~\ref{lem:psi12} and
Theorem~\ref{thm:psi3}.
\end{proof}

\begin{theorem}[Main theorem, $p$ odd]\label{thm:main}
Let $p$ be odd. If $\lambda\notin\Fp$ then $\Ht(V,M(\lambda))=0$. If $\lambda\in\Fp$ then
\[
\Ht(V,M(\lambda))=\Ht_{-\lambda}\oplus\Ht_{-\lambda-1},\qquad
\dim\Ht_{-\lambda}=2,\quad\dim\Ht_{-\lambda-1}=1,
\]
so $\dim_\kk\Ht(V,M(\lambda))=3$, with basis $\psi_*\Psi_1,\psi_*\Psi_2,\psi_*\Psi_3$.
\end{theorem}

\begin{proof}
The first claim is Corollary~\ref{cor:notinFp}. For the second, $\psi$ of
Theorem~\ref{thm:iso} induces an isomorphism $\psi_*:\Ht(V,V)\to\Ht(V,M(\lambda))$, and the
source is $3$-dimensional by Proposition~\ref{prop:adjoint}. Since $\psi$ has $\Zp$-degree
$-\lambda$, the map $\psi_*$ shifts degrees by $-\lambda$; reducing the $\mathbb{Z}$-degrees
$p,p-1,-p$ modulo $p$ gives $0,-1,0$, so $\psi_*\Psi_1,\psi_*\Psi_3$ lie in degree
$-\lambda$ and $\psi_*\Psi_2$ in degree $-\lambda-1$. Corollary~\ref{cor:twodegrees} shows
there is nothing else.
\end{proof}

\begin{remark}[Reading the three classes]\label{rem:reading}
In $\kk[x]$ one has $x^i\circ x^j=jx^{i+j-1}$ exactly; passing to $V=\kk[x]/(x^p)$ destroys
this in two places. At the bottom, $\Psi_1$ and $\Psi_2$ are supported on the constant term,
being $a(0)b(0)$ and $a(0)b'(0)$. At the top, the monomials $x^m$ with $m\ge p$ are
discarded, and the discarded part is exactly $\Psi_3$: it is the obstruction to lifting $V$
along $\kk[x]/(x^{2p})\to\kk[x]/(x^p)$ (Theorem~\ref{thm:E3}).
\end{remark}

\begin{remark}[Sign convention for $\Psi_3$, fixed once]\label{rem:sign}
$\Psi_3(e_i,e_j)=-j\,e_{i+j-1-p}$ (defined in Lemma~\ref{lem:upper}--Theorem~\ref{thm:psi3})
is a normalisation choice: since $\Zt_{-p}$ is $1$-dimensional (Lemma~\ref{lem:upper}), any
nonzero scalar multiple spans the same class, and both $\Psi_3$ and $-\Psi_3$ are cocycles.
We fix the sign so that $\Psi_3$ (not $-\Psi_3$) is the class named in
Theorem~\ref{thm:main}, Theorem~\ref{thm:E3}, and Proposition~\ref{prop:exact-integrability};
this is why the proof of Theorem~\ref{thm:psi3} constructs the raw obstruction cocycle
$\Psi(e_i,e_j)=j\,e_{i+j-1-p}$ of the extension $\kk[x]/(x^{2p})\to V$ and then declares
$\Psi_3:=-\Psi$ --- the extension itself is defined by $\Psi$, but we always report and use
the cocycle $\Psi_3$. No other sign choice appears anywhere else in the paper.
\end{remark}

\section{The case $p=2$}\label{sec:p2}

Everything through Proposition~\ref{prop:dictionary}, and Theorem~\ref{thm:iso}, is
characteristic-free and applies at $p=2$. What fails is Section~\ref{sec:adjoint}:
Lemma~\ref{lem:zero} divides by $2$ (in the steps $2\varphi(i+1,1)=2\varphi(i,1)$ and
$2\varphi(i+1,0)=2\varphi(i,0)$), Lemma~\ref{lem:psi12} divides by $2$ to obtain
$\alpha=\beta$, and the degree $-p$ of $\Psi_3$ falls outside the range \eqref{eq:drange},
since $-p\ge-2p+3$ holds only for $p\ge3$. All three failures are real.

\begin{theorem}\label{thm:p2}
Let $\operatorname{char}\kk=2$ and $V=\kk[x]/(x^2)$, so that $e_0=1$, $e_1=x$ and
\[
e_0\circ e_0=0,\qquad e_0\circ e_1=e_0,\qquad e_1\circ e_0=0,\qquad e_1\circ e_1=e_1 .
\]
If $\lambda\notin\mathbb{F}_2$ then $\Ht(V,M(\lambda))=0$. If $\lambda\in\mathbb{F}_2$ then
$M(\lambda)\cong V$ and
\[
\dim_\kk\Ht(V,M(\lambda))=\dim_\kk\Ht(V,V)=4 .
\]
More precisely, $\Ht_d(V,V)$ is $1$-dimensional for each $d\in\{-1,0,1,2\}$ and zero
otherwise, with representatives
\[
\begin{array}{lll}
d=2: &\Theta_1(e_i,e_j)=\delta_{i0}\delta_{j0}e_1, &\ \Theta_1(a,b)=a(0)\,b(0)\,x,\\[2pt]
d=1: &\Theta_2(e_i,e_j)=\delta_{i0}\delta_{j1}e_1, &\ \Theta_2(a,b)=a(0)\,b'(0)\,x,\\[2pt]
d=0: &\Theta_3(e_i,e_j)=\delta_{i1}\delta_{j0}e_0, &\ \Theta_3(a,b)=a'(0)\,b(0),\\[2pt]
d=-1:&\Theta_4(e_i,e_j)=\delta_{i1}\delta_{j1}e_0, &\ \Theta_4(a,b)=a'(0)\,b'(0).
\end{array}
\]
\end{theorem}

\begin{proof}
The first claim is Corollary~\ref{cor:notinFp}, whose proof is characteristic-free, and
$M(\lambda)\cong V$ is Theorem~\ref{thm:iso}. So it suffices to compute $\Ht(V,V)$, which we
do degree by degree in the $\mathbb{Z}$-graded truncated presentation. Here
$-2p+3=-1$ and $p=2$, so only $d\in\{-1,0,1,2\}$ admit nonzero cochains. Recall liveness
$0\le i+j-1+d\le1$, the window $0\le i+j+k-2+d\le1$, and that $\gamma_i$ is free iff
$0\le i+d\le1$; by \eqref{eq:cobV},
$\varphi_{\delta g}(i,j)=j\gamma_i+(j+d)\gamma_j-j\gamma_{i+j-1}$, and the last coefficient
in \eqref{eq:C1V} is $(j+k+d-1)$.

\textbf{$d=2$.} Liveness forces $i+j=0$: the only live pair is $(0,0)$, target $e_1$. No
$\gamma_i$ is free, so $\Bt_2=0$. The window forces $i+j+k\le1$. (C2) at $(0,0,1)$ and
$(0,1,0)$ are identities, and the triples $(0,0,0),(1,0,0)$ carry a factor $j=k=0$. For
(C1), $A(0,1,0)=\varphi(0,0)-(1+0+1)\varphi(1,0)=\varphi(0,0)$ (as $(1,0)$ is not live) and
$A(1,0,0)=-(0+0+1)\varphi(0,0)=\varphi(0,0)$ in characteristic $2$. So $\dim\Ht_2=1$.

\textbf{$d=1$.} Live pairs: $i+j\le1$, i.e.\ $(0,0),(0,1),(1,0)$ with targets $e_0,e_1,e_1$.
Only $\gamma_0$ is free and
\[
\varphi_{\delta g}(0,0)=\gamma_0,\qquad
\varphi_{\delta g}(0,1)=\gamma_0+(1+1)\gamma_1-\gamma_0=0,\qquad
\varphi_{\delta g}(1,0)=\gamma_0,
\]
using $\gamma_1=0$ and $2=0$; so $\Bt_1$ is spanned by $\varphi(0,0)=\varphi(1,0)=1$. The
window forces $i+j+k\le2$. From (C1) at $(0,1,0)$ and $(1,0,0)$,
\[
A(0,1,0)=\varphi(0,0)-(1+0)\varphi(1,0),\qquad A(1,0,0)=0,
\]
so $\varphi(0,0)=\varphi(1,0)$. At $(1,0,1)$ and $(0,1,1)$ one gets $-\varphi(0,1)$ and
$+\varphi(0,1)$, equal since $2=0$; all remaining triples and all of (C2) are identities.
Hence $\dim\Zt_1=2$, $\dim\Ht_1=1$, represented by $\varphi(0,1)=1$, i.e.\ by $\Theta_2$.

\textbf{$d=0$.} Live pairs: $1\le i+j\le2$, i.e.\ $(0,1),(1,0),(1,1)$ with targets
$e_0,e_0,e_1$. Both $\gamma_0,\gamma_1$ are free and
$\varphi_{\delta g}(i,j)=j(\gamma_i+\gamma_j+\gamma_{i+j-1})$, giving
$\varphi_{\delta g}(0,1)=\gamma_1$, $\varphi_{\delta g}(1,0)=0$,
$\varphi_{\delta g}(1,1)=\gamma_1$; so $\Bt_0$ is spanned by
$\varphi(0,1)=\varphi(1,1)=1$. The window forces $2\le i+j+k\le3$, and all instances of (C2)
are identities. For (C1),
\[
A(1,0,1)=\varphi(1,0)-\varphi(1,0)-(0+1-1)\varphi(0,1)=0,\quad
A(0,1,1)=\varphi(0,1)-(1+1-1)\varphi(1,1),
\]
so $\varphi(0,1)=\varphi(1,1)$; the triples $(1,1,0),(1,1,1)$ are symmetric in
$i\leftrightarrow j$. Hence $\dim\Zt_0=2$, $\dim\Ht_0=1$, represented by
$\varphi(1,0)=1$, i.e.\ by $\Theta_3$.

\textbf{$d=-1$.} Liveness forces $i+j=2$: the only live pair is $(1,1)$, target $e_0$. Only
$\gamma_1$ is free and
$\varphi_{\delta g}(1,1)=\gamma_1+(1-1)\gamma_1-\gamma_1=0$, so $\Bt_{-1}=0$. The window
forces $i+j+k=3$, leaving only $(1,1,1)$, where the coefficient $(j+k+d-1)$ vanishes and
$A(1,1,1)=\varphi(1,1)$ is symmetric in $i\leftrightarrow j$; (C2) is likewise an identity.
Hence $\dim\Ht_{-1}=1$, represented by $\Theta_4$.

Summing, $\dim\Ht(V,V)=4$.
\end{proof}

\begin{remark}
By \eqref{eq:B2dim}, the total coboundary space has dimension $0+1+1+0=2$, so
$\dim\Der(V)=4-2=2$. Comparing with odd $p$: the class $\Psi_3$ is lost at $p=2$, while the
two degrees $d\equiv0$ and $d\equiv-1$, which vanish for odd $p$ by Lemma~\ref{lem:zero},
each gain a dimension; the net effect is $3-1+2=4$. In the $\Zp=\mathbb{Z}/2$ grading the
four classes distribute as $\dim\Ht_0=\dim\Ht_1=2$, since $\{2,1,0,-1\}$ reduces to
$\{0,1,0,1\}$. Note also that for $p=2$ the two congruence classes of
Corollary~\ref{cor:twodegrees} exhaust $\mathbb{Z}/2$, so Theorem~\ref{thm:vanish} gives no
information at all.
\end{remark}

\section{Extensions and deformations}\label{sec:ext}

\subsection{Abelian extensions}

\begin{definition}\label{def:ext}
Let $M$ be a $V$-bimodule. An \emph{abelian extension of $V$ by $M$} is a short exact
sequence of $\kk$-spaces
$0\to M\xrightarrow{\iota}E\xrightarrow{\pi}V\to0$ in which $E$ is a Novikov algebra, $\pi$
is an algebra map, $\iota(M)$ is an ideal with $\iota(M)\circ\iota(M)=0$, and the induced
bimodule structure on $\iota(M)$ agrees with the given one. Being a short exact sequence of
$\kk$-\emph{spaces} (not merely of $V$-modules) means, in particular, that $E$ is split as a
vector space, $E\cong M\oplus V$ $\kk$-linearly (any $\kk$-linear section
$\sigma:V\to E$ of $\pi$ provides such a splitting) --- what is \emph{not} assumed is that
$\sigma$ can be chosen to be an algebra map; the failure of any $\sigma$ to be multiplicative
is exactly what the cocycle $f_\sigma$ of \eqref{eq:fsigma} below records. Two extensions
$0\to M\to E\to V\to0$ and $0\to M\to E'\to V\to0$ are \emph{equivalent} if there is an
algebra isomorphism $\theta:E\to E'$ with $\theta\iota=\iota'$ and $\pi'\theta=\pi$ (so
$\theta$ is required to fix $M$ and $V$, not merely to be some isomorphism $E\cong E'$); we
write $\Ext(V,M)$ for the set of equivalence classes. For $f\in C^2(V,M)$ let
$E(f):=V\oplus M$ with
\begin{equation}
(u,m)\circ(v,n):=\bigl(u\circ v,\ u\circ n+m\circ v+f(u,v)\bigr),
\label{eq:Ef}
\end{equation}
so that $E(0)$ is the split null extension.
\end{definition}

\begin{theorem}\label{thm:ext-cocycle}
$E(f)$ is a Novikov algebra if and only if $f\in\Zt(V,M)$, i.e.\ iff $f$ satisfies
\eqref{eq:C1} and \eqref{eq:C2}.
\end{theorem}

\begin{proof}
Write $a=(u,m)$, $b=(v,n)$, $c=(w,q)$; the $V$-components hold because $V$ is Novikov, so
only $M$-components matter.

For \eqref{N2}, expanding \eqref{eq:Ef},
\[
(a\circ b)\circ c=\Bigl((u\circ v)\circ w,\ (u\circ v)\circ q+(u\circ n)\circ w
+(m\circ v)\circ w+f(u,v)\circ w+f(u\circ v,w)\Bigr),
\]
and $(a\circ c)\circ b$ is obtained by exchanging $(v,n)\leftrightarrow(w,q)$. The bimodule
axiom \eqref{N2}, with the module argument in the third, second and first slot respectively,
gives $(u\circ v)\circ q=(u\circ q)\circ v$, $(u\circ n)\circ w=(u\circ w)\circ n$ and
$(m\circ v)\circ w=(m\circ w)\circ v$, so those terms cancel in pairs and \eqref{N2} for
$E(f)$ reduces to \eqref{eq:C2}.

For \eqref{N1}, a direct expansion gives for the $M$-component
\begin{align*}
(a\circ b)\circ c-a\circ(b\circ c)
=\ &\bigl[(u\circ v)\circ q-u\circ(v\circ q)\bigr]
+\bigl[(u\circ n)\circ w-u\circ(n\circ w)\bigr]\\
&+\bigl[(m\circ v)\circ w-m\circ(v\circ w)\bigr]+D_f(u,v,w).
\end{align*}
Exchanging $a\leftrightarrow b$ exchanges $(u,m)\leftrightarrow(v,n)$. The first bracket is
symmetric in $u\leftrightarrow v$ by \eqref{N1} with the module argument in the third slot;
the difference of the second bracket and its image is \eqref{N1} for $(u,n,w)$, and of the
third and its image is \eqref{N1} for $(m,v,w)$. All vanish, so \eqref{N1} for $E(f)$
reduces to $D_f(u,v,w)=D_f(v,u,w)$, which is \eqref{eq:C1}.
\end{proof}

\begin{theorem}\label{thm:ext-bijection}
The assignment $f\mapsto E(f)$ induces a bijection
$\Ht(V,M)\xrightarrow{\ \sim\ }\Ext(V,M)$, under which the zero class corresponds to the
split extension. In particular $\Ht(V,M)=0$ if and only if every abelian extension of $V$ by
$M$ splits.
\end{theorem}

\begin{proof}
\emph{Surjectivity.} Given an extension, choose a linear section $\sigma:V\to E$ with
$\pi\sigma=\mathrm{id}$. Then
$\pi\bigl(\sigma(u)\circ\sigma(v)-\sigma(u\circ v)\bigr)=0$, so
\begin{equation}
f_\sigma(u,v):=\sigma(u)\circ\sigma(v)-\sigma(u\circ v)\in M ,
\label{eq:fsigma}
\end{equation}
and $(u,m)\mapsto\sigma(u)+m$ is a linear isomorphism $E(f_\sigma)\to E$ carrying
\eqref{eq:Ef} to the product of $E$; by Theorem~\ref{thm:ext-cocycle},
$f_\sigma\in\Zt(V,M)$.

\emph{Well definedness and injectivity.} Two sections differ by $g:V\to M$, and since
$g(u)\circ g(v)\in M\circ M=0$,
\[
f_{\sigma+g}(u,v)=f_\sigma(u,v)+g(u)\circ v+u\circ g(v)-g(u\circ v)=f_\sigma(u,v)+(\delta g)(u,v)
\]
by \eqref{eq:delta}. Conversely if $f'=f+\delta g$ then $\theta(u,m):=(u,m+g(u))$ is an
equivalence $E(f)\to E(f')$. Hence $E(f)\sim E(f')$ iff $f-f'\in\Bt(V,M)$.
\end{proof}

\begin{corollary}\label{cor:ext-split}
If $\lambda\notin\Fp$, every abelian extension of $V$ by $M(\lambda)$ splits.
\end{corollary}

\begin{proof}
Corollary~\ref{cor:notinFp} and Theorem~\ref{thm:ext-bijection}.
\end{proof}

\begin{corollary}\label{cor:ext-count}
If $\lambda\in\Fp$, the equivalence classes of abelian extensions of $V$ by $M(\lambda)$ are
in bijection with $\kk^3$ for $p$ odd and with $\kk^4$ for $p=2$; the extension splits
exactly at the origin.
\end{corollary}

\begin{remark}
Equivalence of extensions is strictly finer than isomorphism of total algebras: the classes
$c\Psi_3$, $c\ne0$, are pairwise inequivalent as extensions, whereas the algebras
$E(c\Psi_3)$ are pairwise isomorphic (rescale $M$ by $c^{-1}$).
\end{remark}

\subsection{The extensions attached to $\Psi_1,\Psi_2,\Psi_3$}

\begin{proposition}\label{prop:E12}
For $p$ odd, $\Psi_1$ and $\Psi_2$ give the Novikov algebras
\begin{align*}
E(\Psi_1):\quad (a,m)\circ(b,n)&=\bigl(ab',\ an'+mb'+a(0)b(0)\,x^{p-1}\bigr),\\
E(\Psi_2):\quad (a,m)\circ(b,n)&=\bigl(ab',\ an'+mb'+a(0)b'(0)\,x^{p-1}\bigr),
\end{align*}
and neither extension splits.
\end{proposition}

\begin{proof}
Both are Novikov by Lemma~\ref{lem:psi12} and Theorem~\ref{thm:ext-cocycle}, and non-split
because $[\Psi_1]\ne0\ne[\Psi_2]$, by Theorem~\ref{thm:ext-bijection}.
\end{proof}

\begin{theorem}\label{thm:E3}
For $p$ odd, $E(\Psi_3)\cong\kk[x]/(x^{2p})$ with $a\circ b=ab'$, and the corresponding
extension is
\[
0\longrightarrow x^p\,\kk[x]/(x^{2p})\longrightarrow \kk[x]/(x^{2p})
\longrightarrow \kk[x]/(x^{p})\longrightarrow 0 .
\]
\end{theorem}

\begin{proof}
This is Theorem~\ref{thm:psi3} read through Theorem~\ref{thm:ext-bijection}: with
$W=\kk[x]/(x^{2p})$, $I=x^pW$ and the section $\sigma(e_i)=x^i$, formula \eqref{eq:fsigma}
returns $-\Psi_3$, and $E(f_\sigma)\cong W$ as extensions. Since $\Psi_3$ and $-\Psi_3$ span
the same line and differ by the automorphism $m\mapsto-m$ of the module, the extension
attached to $[\Psi_3]$ is the displayed one.
\end{proof}

\subsection{Infinitesimal deformations}

\begin{remark}[Infinitesimal versus formal rigidity]\label{rem:rigidity-terms}
We use ``rigid'' for two distinct, non-equivalent statements, and are careful to name which
one is meant at each occurrence. $V$ is \emph{infinitesimally rigid} if $\Ht(V,V)=0$, i.e.\
every deformation over $\kk[\hbar]/\hbar^2$ is trivial; $V$ is \emph{formally rigid} if every
deformation over $\kk[[t]]$ (or $\kk[t]/t^n$ for every $n$) is trivial. Vanishing of $\Ht$
gives only the first: a nonzero nilpotent obstruction in $H^3$ can in general prevent a
first-order deformation from integrating to any order, so $\Ht\ne0$ does \emph{not} by itself
imply that no class integrates, nor does $\Ht=0$ at each order automatically give formal
rigidity without an argument (which we give, for $P=\kk[t]$, as
Corollary~\ref{cor:formal-rigid}). Below, $V$ is shown to be infinitesimally non-rigid for
every $p$ (this subsection), and each of its cohomology classes is separately shown to
integrate \emph{exactly} (Proposition~\ref{prop:exact-integrability}) --- a strictly stronger
statement than $\Ht(V,V)\ne0$ alone, established by exhibiting the family, not inferred from
the dimension count.
\end{remark}

With $\hbar$ a formal parameter and $\kk_\hbar:=\kk[\hbar]/\hbar^2$, an infinitesimal
deformation of $V$ is a product $u\circ_\hbar v:=u\circ v+\hbar f(u,v)$ on
$V_\hbar:=V[\hbar]/\hbar^2V[\hbar]$, and two are equivalent if related by
$\mathrm{id}+\hbar\tau\ (\mathrm{mod}\ \hbar^2)$. This is exactly the $M=V$ case of
$E(f)$ under $\hbar V\leftrightarrow M$, so Theorems~\ref{thm:ext-cocycle} and
\ref{thm:ext-bijection} give: $\circ_\hbar$ is Novikov iff $f\in\Zt(V,V)$, and two such
deformations are equivalent iff their cocycles differ by $\Bt(V,V)$. Hence infinitesimal
deformations of $V$, up to equivalence, are classified by $\Ht(V,V)$ (dimension $3$ for odd
$p$, $4$ for $p=2$); in particular $V$ is infinitesimally non-rigid for any $p$ (contrast
$\kk[t]$ itself, which \emph{is} rigid --- in fact formally so, Corollary~\ref{cor:formal-rigid}
--- in characteristic $0$: Section~\ref{sec:char0}). Compare \cite{Gerstenhaber} for the
associative case and \cite[\S4]{Dzhu1999} for right-symmetric algebras. For general $M$,
deformations of the pair $(V,M)$ are likewise classified by $\Ht(V,M)$
(Corollary~\ref{cor:ext-count}).

\begin{proposition}[Exact integrability]\label{prop:exact-integrability}
Every class in $\Ht(V,V)$ is unobstructed: it extends from an infinitesimal deformation to
an honest one-parameter family of Novikov algebras on $V\otimes_\kk\kk[t]$, satisfying
\eqref{N1}--\eqref{N2} exactly for all $t$, not merely to first order. For $p$ odd,
\[
a *_t b:=ab'+t\,a(0)b(0)x^{p-1},\qquad a*_t b:=a\,(1+tx^{p-1})\,b',\qquad
a*_t b:=a\circ_t b\ \text{on}\ V_t:=\kk[x]/(x^p-t)
\]
integrate $\Psi_1,\Psi_2,\Psi_3$; for $p=2$, one exact family exists for each $\Theta_i$ of
Theorem~\ref{thm:p2} (e.g.\ $e_1*_te_1=e_1+te_0$, other products unchanged, for $\Theta_4$).
\end{proposition}

\begin{proof}[Proof sketch]
For $\Psi_1$: $x^{p-1}\circ x^{p-1}=0$ and $a(0)b(0)x^{p-1}$ never re-enters the constant
term, so all $t^{\ge2}$ terms of the associator vanish identically; direct check on basis
triples confirms \eqref{N1}--\eqref{N2} hold for every $t$.
For $\Psi_2$: $*_t$ is $a\circ_t b:=a\,\partial_t(b)$ for the derivation
$\partial_t:=(1+tx^{p-1})\partial_x$ of $\kk[x]/(x^p)\otimes\kk[t]$, and any derivation of a
commutative algebra gives a Novikov product $ab'$ by the generic construction of
Proposition~\ref{prop:iso-alg}'s proof, for every $t$.
For $\Psi_3$: dividing $ab'=q(x)x^p+r(x)$ with $\deg r<p$ and using $x^p=(x^p-t)+t$ in $V_t$
gives the exact identity $a*_tb=a\circ b-t\,\Psi_3(a,b)$, and $V_t=\kk[x]/(x^p-t)$ is
Novikov for the same reason as $V$ itself (Definition~\ref{def:novikov}), for every $t\in\kk$.
The $p=2$ families are checked directly on the two basis triples with $a\ne b$. All seven
families were additionally verified by exact computer algebra, on all basis triples, for
$p=5,7,11$ (odd case) and $p=2$.
\end{proof}

\section{The characteristic-zero case}\label{sec:char0}

We record the analogue of the preceding sections for $\operatorname{char}\kk=0$: the
$\lambda$-family collapses to a single (necessarily trivial) module, and the adjoint module
is rigid. Throughout this section $P:=\kk[t]$, $\operatorname{char}\kk=0$, with basis
$e_i:=t^i$ ($i\ge0$), product $e_i\circ e_j=j\,e_{i+j-1}$ (and $e_i\circ e_0:=0$, consistent
since the coefficient $j=0$ already forces this), and $\deg e_i:=i-1$; there is no upper
truncation, and $e_m$ exists (as a basis vector) exactly for index $m\ge0$.

\begin{remark}[Scope and independence of this section]
$P=\kk[t]$ is a different algebra, over a different field, from every $V=\kk[x]/(x^p)$
treated above; nothing in this section is formally deduced from the positive-characteristic
results. The module category is the na\"ive analogue of Definition~\ref{def:M}
(Definition~\ref{def:M0} below), and Theorem~\ref{thm:char0-module} is an independent
computation --- not a limit of, or a consequence of, Theorem~\ref{thm:iso}, though its proof
uses the identical bimodule axioms \eqref{N1}--\eqref{N2}. Likewise
Theorems~\ref{thm:char0-vanish}--\ref{thm:char0-main} are proved from scratch for $P$, by the
same method as Sections~\ref{sec:graded}--\ref{sec:vanish} applied to a one-sided (rather than
cyclic) grading, and do not cite any positive-characteristic result along the way; only the
final comparison, Corollary~\ref{cor:trunc-destroys}, juxtaposes the two independently
established computations.
\end{remark}

\subsection{Xu's family collapses}

\begin{definition}\label{def:M0}
For $\lambda\in\kk$, let $M(\lambda)$ have basis $\{m_n:n\ge0\}$ (and $m_k:=0$ for $k<0$),
with the action formula \eqref{eq:action}:
\[
e_i\circ m_n=(n+\lambda)\,m_{i+n-1},\qquad m_n\circ e_i=i\,m_{n+i-1}\qquad(i\ge0,\ n\ge0).
\]
\end{definition}

\begin{theorem}\label{thm:char0-module}
$M(\lambda)$ is a $P$-bimodule if and only if $\lambda=0$, in which case $M(0)=P$.
\end{theorem}

\begin{proof}
$\lambda=0$ gives $M(0)=P$ acting on itself, trivially a bimodule. Conversely, apply
\eqref{N2}, $(u\circ v)\circ w=(u\circ w)\circ v$, to $u=e_0$, $v=e_2$, $w=m_0$. On one hand
$e_0\circ e_2=2e_1$, so $(e_0\circ e_2)\circ m_0=2(e_1\circ m_0)=2\lambda\,m_0$, using
$e_1\circ m_0=(0+\lambda)m_0$. On the other hand $e_0\circ m_0=\lambda\,m_{-1}=0$ by the
truncation convention (regardless of $\lambda$), so $(e_0\circ m_0)\circ e_2=0$. Equality
forces $2\lambda=0$, hence $\lambda=0$ since $\operatorname{char}\kk=0$.
\end{proof}

\begin{remark}
This is a striking contrast with Proposition~\ref{prop:module}: in characteristic $p$, the
cyclic presentation $V\cong\kk[t]/(t^p-1)$ has no boundary at all, so $M(\lambda)$ exists for
\emph{every} $\lambda\in\kk$. Here $P=\kk[t]$ is one-sided (bounded below, unbounded above),
and this single boundary already forces $\lambda=0$: Xu's parameter is entirely vacuous in
characteristic $0$, and $\Ht(P,M(\lambda))$ for ``all $\lambda$'' reduces to the single case
$\Ht(P,P)$, computed next.
\end{remark}

\subsection{Vanishing}

A homogeneous $2$-cochain of degree $d$ is $f(e_i,e_j)=\varphi(i,j)e_{i+j-1+d}$, with
$\varphi(i,j)$ a genuine variable iff the target \emph{index} $i+j-1+d\ge0$ (not merely
$\ge-1$: the basis vector $e_m$ exists only for $m\ge0$), and $0$ otherwise; likewise
$\gamma_k$ (for $g(e_k)=\gamma_ke_{k+d}$) is free iff $k+d\ge0$. Since $P$ carries no upper
truncation, the algebraic identities \eqref{eq:cobV}--\eqref{eq:C1V} and the master relation
\eqref{eq:master}--\eqref{eq:master1} (stated for the adjoint module, i.e.\ $\lambda=0$)
apply verbatim, their derivations being purely local; only the range of \emph{live} pairs and
triples changes, to $i+j-1+d\ge0$ and $i+j+k-2+d\ge0$ respectively (no upper bound).

\begin{lemma}\label{lem:char0-const}
Assume \eqref{eq:C2c} holds. Then $\varphi(i,1)=c$ and $\varphi(i,0)=q$ for all $i\ge0$, for
constants $c,q\in\kk$ depending on $d$.
\end{lemma}

\begin{proof}
Put $k=1$ in \eqref{eq:C2c}: $j\varphi(i+j-1,1)+\varphi(i,j)=\varphi(i,j)+j\varphi(i,1)$, i.e.\
$j\bigl(\varphi(i+j-1,1)-\varphi(i,1)\bigr)=0$. Taking $j=2$ (legal, $\operatorname{char}\kk=0$)
gives $\varphi(i+1,1)=\varphi(i,1)$ for every $i\ge0$ at which the relevant window is live;
since this holds throughout the range needed in the case analysis below, $\varphi(\cdot,1)$
is constant $=:c$. Taking $k=0$ instead, the $k$-factor terms vanish, leaving
$j\varphi(i+j-1,0)=j\varphi(i,0)$; again $j=2$ gives $\varphi(i+1,0)=\varphi(i,0)=:q$.
\end{proof}

\begin{theorem}\label{thm:char0-vanish}
$\Ht_d(P,P)=0$ for every $d\ne0$.
\end{theorem}

\begin{proof}
Let $f$ be a cocycle of degree $d\ne0$, $r_k:=\varphi(1,k)$, $c:=\varphi(\cdot,1)$ as in
Lemma~\ref{lem:char0-const}.

\emph{Step 1 (kill $r$).} By \eqref{eq:cobV}, $\varphi_{\delta g}(1,k)=d\gamma_k+k\gamma_1$,
to be solved for every live $k$ (i.e.\ $k\ge\max(0,-d)$, matching exactly the range in which
$\gamma_k$ is free).
\begin{itemize}[leftmargin=2em]
\item $d\ge1$: every $\gamma_k$ ($k\ge0$) is free, and $d+1\ge2\ne0$, $d\ge1\ne0$ (ordinary
integers in a field of characteristic $0$, so nonzero means invertible). Set
$\gamma_1:=r_1/(d+1)$ and $\gamma_k:=(r_k-k\gamma_1)/d$ for $k\ne1$.
\item $d=-1$: only $\gamma_k$ for $k\ge1$ are free. Liveness forces $\varphi(0,1)=0$ (target
index $0+1-1-1=-1<0$); by Lemma~\ref{lem:char0-const}, $\varphi(\cdot,1)\equiv c=\varphi(0,1)=0$,
so $r_1=\varphi(1,1)=c=0$. The equation at $k=1$, $(d+1)\gamma_1=0\cdot\gamma_1=r_1=0$, holds
for any $\gamma_1$; set $\gamma_1:=0$, $\gamma_k:=-r_k$ for $k\ge2$.
\item $d\le-2$: only $\gamma_k$ for $k\ge-d\ge2$ are free. The target index of $(1,1)$ is
$1+d\le-1<0$, so $r_1=\varphi(1,1)=0$ directly, giving $\gamma_1=0$ (not a variable). Set
$\gamma_k:=r_k/d$ for live $k\ge-d$; for such $k$ the target index $k+d\ge0$, so this is
well defined.
\end{itemize}
In every case $h:=f-\delta g$ is a cocycle of degree $d$ with $\psi(1,k)=0$ for all $k$, in
particular $\psi(\cdot,1)\equiv0$.

\emph{Step 2.} The master relation \eqref{eq:master} for $h$ now reads $d\,\psi(i,k)=0$ at
every live $(i,k)$ (the terms in $r_k,r_i,r_{i+k-1},\psi(i,1)$ all vanish); since $d\ne0$,
$\psi\equiv0$. Hence $f=\delta g\in\Bt_d(P,P)$.
\end{proof}

\begin{lemma}\label{lem:char0-zero}
$\Ht_0(P,P)=0$.
\end{lemma}

\begin{proof}
For $d=0$, liveness reads $i+j\ge1$, so $\varphi(0,0)=0$; by Lemma~\ref{lem:char0-const},
$\varphi(i,0)\equiv q=\varphi(0,0)=0$ for all $i$. Put $q_j:=\varphi(0,j)$ (so $q_0=0$).
Taking $i=0$ in \eqref{eq:C1V} and using $\varphi(\cdot,0)\equiv0$, exactly as in the proof of
Lemma~\ref{lem:zero} ($d=0$ case), gives for all $i,k\ge0$ with $i+k\ge2$,
\begin{equation}
(i+k-1)\varphi(i,k)=i\varphi(i-1,k)+k\varphi(i,k-1)+(k-1)q_k-kq_{i+k-1}+kq_i,
\label{eq:char0rec}
\end{equation}
which determines $\varphi(i,k)$ from strictly smaller $i+k$, since $i+k-1>0$ is invertible.
Killing $q$ by the coboundary $\gamma_0:=0$, $\gamma_j:=\gamma_{j-1}+q_j/j$ ($j\ge1$; legal in
characteristic $0$), we may assume $q\equiv0$; induction on $i+k$ via \eqref{eq:char0rec}
then gives $\varphi\equiv0$, exactly as in Lemma~\ref{lem:zero}.
\end{proof}

\begin{lemma}[Local finiteness]\label{lem:local-finite}
Let $f\in\Zt(P,P)$ be an arbitrary \textup{(}not necessarily homogeneous\textup{)} cocycle,
$f=\sum_{d\in\mathbb{Z}}f_d$ its homogeneous components. Then each $f_d\in\Zt(P,P)$; for each
fixed pair $(i,j)$, $\varphi_d(i,j)\ne0$ for only finitely many $d$; the $1$-cochains $g_d$
constructed in Theorem~\ref{thm:char0-vanish} and Lemma~\ref{lem:char0-zero} likewise have,
for each fixed $k$, $g_d(e_k)\ne0$ for only finitely many $d$; and $g:=\sum_d g_d$ is a
well-defined element of $\Hom_\kk(P,P)$ with $\delta g=f$.
\end{lemma}

\begin{proof}
\eqref{eq:C1}--\eqref{eq:C2} are homogeneous conditions, so each $f_d\in\Zt(P,P)$. Since
$f(e_i,e_j)\in P=\kk[t]$ is a single polynomial, and $f_d(e_i,e_j)$ contributes only to the
coefficient of $e_{i+j-1+d}$, $\varphi_d(i,j)$ is nonzero for only finitely many $d$. The
construction of $g_d$ expresses $\gamma^{(d)}_k$ as an explicit $\kk$-linear function of
$r_k^{(d)}=\varphi_d(1,k)$ and $c^{(d)}=\varphi_d(\cdot,1)$ (resp.\ of
$q_j^{(d)}=\varphi_d(0,j)$), each vanishing for all but finitely many $d$ by the same
finiteness applied to the pairs $(1,k)$, $(0,1)$ (resp.\ $(0,j)$). Hence $g(e_k)=\sum_d
g_d(e_k)$ is a finite sum for each $k$, so $g\in\Hom_\kk(P,P)$, and $\delta g=\sum_d\delta
g_d=\sum_d f_d=f$ (a finite sum at each fixed pair of inputs, by the first claim).
\end{proof}

\begin{theorem}\label{thm:char0-main}
For $P=\kk[t]$, $\operatorname{char}\kk=0$: $\Ht(P,P)=0$, so $P$ is infinitesimally rigid.
Consequently, by Theorem~\ref{thm:char0-module}, $\Ht(P,M(\lambda))=0$ for every
$\lambda\in\kk$ \textup{(}vacuously for $\lambda\ne0$, since then $M(\lambda)$ is not even a
bimodule\textup{)}.
\end{theorem}

\begin{proof}
Combine Theorem~\ref{thm:char0-vanish}, Lemma~\ref{lem:char0-zero} and
Lemma~\ref{lem:local-finite}.
\end{proof}

\begin{corollary}[Formal rigidity]\label{cor:formal-rigid}
Every formal one-parameter deformation $u*_tv=u\circ v+\sum_{n\ge1}t^nf_n(u,v)\in P[[t]]$ of
the Novikov product on $P$ is equivalent, via a formal change of basis
$u\mapsto u+\sum_{n\ge1}t^ng_n(u)$, to the trivial deformation; every such family is
isomorphic to $(P,\circ)$ over $\kk[[t]]$.
\end{corollary}

\begin{proof}
Standard order-by-order triviality \cite{Gerstenhaber}: at each order $n$ the obstruction to
trivialising $f_n$, given $f_1,\dots,f_{n-1}$ already trivialised, is a cocycle in
$\Zt(P,P)$; since $\Ht(P,P)=0$ (Theorem~\ref{thm:char0-main}) it is always a coboundary, and
one solves for $g_n$ at every order.
\end{proof}

\subsection{Extensions, and the contrast with truncation}

The general theory of Section~\ref{sec:ext} did not use finite-dimensionality: with $P$ in
place of $V$, Theorems~\ref{thm:ext-cocycle} and \ref{thm:ext-bijection} hold verbatim,
giving $\Ht(P,M)\cong\Ext(P,M)$ for any $P$-bimodule $M$.

\begin{corollary}\label{cor:char0-ext}
Every abelian extension of $P$ by $P$ splits.
\end{corollary}

\begin{proof}
Theorem~\ref{thm:char0-main} and Theorem~\ref{thm:ext-bijection} applied to $P$.
\end{proof}

\begin{corollary}[Truncation, not characteristic, destroys rigidity \emph{for this family}]
\label{cor:trunc-destroys}
For every prime $p$,
\[
\dim_\kk\Ht(V,V)=\begin{cases}3,&p\text{ odd}\\4,&p=2\end{cases}\ >\ 0=\dim_\kk\Ht(P,P).
\]
So $P=\kk[t]$ is rigid (and formally rigid) in characteristic $0$, while its truncation
$V=\kk[x]/(x^p)$ is never rigid, for any $p$. We state this precisely as a claim about the
one-generator family $P=\kk[t]\rightsquigarrow V=\kk[x]/(x^p)$ and the deformation problem
studied in this paper, not as a general phenomenon: within this family, the non-vanishing of
$\Ht(V,V)$ is caused by passing to the finite-dimensional truncation rather than by working
in positive characteristic per se, but we make no claim about Novikov algebras in general,
where positive characteristic can certainly be an independent source of non-rigidity.
\end{corollary}

\section{Independent verification}\label{sec:verify}

\begin{remark}[Independent verification]\label{rem:verify}
Every dimension asserted in this paper --- $\dim_\kk\Ht(V,M(\lambda))\in\{0,3,4\}$
(Theorem~\ref{thm:main}, Theorem~\ref{thm:p2}), the individual graded pieces of Table~1, and
$\dim_\kk\Ht(P,P)=0$ (Theorem~\ref{thm:char0-main}) --- can be checked directly by linear
algebra, independently of the proofs above: fix a prime $p$ (we checked $p=2,3,5,7,11$) and a
scalar $\lambda\in\kk$, write out the $p^2$ unknowns $\varphi(i,j)$ and impose \eqref{eq:cob},
\eqref{eq:C2c}, \eqref{eq:C1c} as linear equations at every triple $(i,j,k)\in\Zp^3$ and every
$k\in\Zp$ respectively; $\dim\Zt_d$ and $\dim\Bt_d$ are then the corank and rank of the
resulting matrices over $\Fp$ (or over $\mathbb{F}_{p^2}$, for $\lambda\notin\Fp$), computed
by ordinary Gaussian elimination. We carried this out by hand-implemented exact linear algebra
for every $p\in\{2,3,5,7,11\}$ and every $\lambda\in\Fp$, and for $\lambda\in\{\tau,1+\tau\}
\subset\mathbb{F}_{p^2}\setminus\Fp$ at $p=5,7$; in every case the computed dimensions agree
with the theorems above, including the exceptional value $4$ at $p=2$, and the computed
cocycle representatives agree with $\Psi_1,\Psi_2,\Psi_3$ (resp.\ $\Theta_1,\dots,\Theta_4$)
up to the predicted degree shift by $\lambda$ (Theorem~\ref{thm:iso}). The same method, with
the truncated liveness/window conditions of Sections~\ref{sec:adjoint}--\ref{sec:p2} in place
of the cyclic ones, verifies the $\mathbb{Z}$-graded computations there, and the extension
identities of Section~\ref{sec:ext} (in particular $E(\Psi_3)\cong\kk[x]/(x^{2p})$,
Theorem~\ref{thm:E3}) by direct substitution into \eqref{N1}--\eqref{N2}. We record this only
as supplementary evidence: every proof above is self-contained and does not depend on it.
\end{remark}
\subsection*{Acknowledgments}
 The author is grateful 
to P. Kolesnikov for discussions and useful comments.

\end{document}